\documentclass[10.5pt,reqno]{article}
\usepackage{mathrsfs}
\usepackage{amsmath,amssymb,amsthm}
\newtheorem{theorem}{Theorem}[section]
\newtheorem{lemma}[theorem]{Lemma}

\newtheorem{corollary}[theorem]{Corollary}

\newtheorem{remark}[theorem]{Remark}
\usepackage[english]{babel}
\numberwithin{equation}{section}
\title{\bf Continuous dependence of the Cauchy problem for the inhomogeneous biharmonic NLS equation in Sobolev spaces}
\author{{JinMyong An, YuIl Jo and JinMyong Kim$^*$}\\
\footnotesize{Faculty of Mathematics, {\bf Kim Il Sung} University, Pyongyang, Democratic People's Republic of Korea}\\%address
\footnotesize{$^*$ Corresponding Author: JinMyong Kim(jm.kim0211@ryongnamsan.edu.kp)}}
\date{}
\begin{document}
\maketitle
\begin{abstract}
In this paper, we study the continuous dependence of the Cauchy problem for the inhomogeneous biharmonic nonlinear Schr\"{o}dinger (IBNLS) equation
\[iu_{t} +\Delta^{2} u=\lambda |x|^{-b}|u|^{\sigma}u,~u(0)=u_{0} \in H^{s} (\mathbb R^{d}),\]
in the standard sense in $H^s$, i.e. in the sense that the local solution flow is continuous $H^s\to H^s$. Here $d\in \mathbb N$, $s>0$, $\lambda\in \mathbb R$ and $\sigma>0$.
To arrive at this goal, we first obtain the estimates of the term $f(u)-f(v)$ in the fractional Sobolev spaces which generalize the similar results of An-Kim \cite{AKC22} and Dinh \cite{D18I}, where $f(u)$ is a nonlinear function that behaves like $\lambda |u|^{\sigma}u$ with $\lambda\in \mathbb R$. These estimates are then applied to obtain the standard continuous dependence result for IBNLS equation with $0<s <\min \{2+\frac{d}{2},\frac{3}{2}d\}$, $0<b<\min\{4,d,\frac{3}{2}d-s,\frac{d}{2}+2-s\}$ and $0<\sigma< \sigma_{c}(s)$, where $\sigma_{c}(s)=\frac{8-2b}{d-2s}$ if $s<\frac{d}{2}$, and $\sigma_{c}(s)=\infty$ if $s\ge \frac{d}{2}$. Our continuous dependence result generalizes that of Liu-Zhang \cite{LZ212} by extending the validity of $s$ and $b$.
\end{abstract}



%%%%%%%%%%%%%%%%%%%%%%%%%%%%%%%%%%%%%%%%%%%%%%%%%%%%%%%%%%%%%%%%%%%%%%%%%%%%%
\section{Introduction}\label{sec 1.}

In this paper, we consider the Cauchy problem for the inhomogeneous biharmonic nonlinear Schr\"{o}dinger (IBNLS) equation
\begin{equation} \label{GrindEQ__1_1_}
\left\{\begin{array}{l} {iu_{t} +\Delta^{2}u=\lambda |x|^{-b} |u|^{\sigma} u,~(t,x)\in \mathbb R\times \mathbb R^{d},}\\
{u(0,x)=u_{0}(x) \in H^{s}(\mathbb R^{d})}, \end{array}\right.
\end{equation}
where $d\in \mathbb N$, $s\ge 0$, $0<b<4$, $\sigma>0$ and $\lambda \in \mathbb R$.
The equation \eqref{GrindEQ__1_1_} may be seen as an inhomogeneous version of the classic biharmonic nonlinear Schr\"{o}dinger equation (also called fourth-order NLS equation),
\begin{equation} \label{GrindEQ__1_2_}
iu_{t} +\Delta^{2}u=\lambda |u|^{\sigma} u.
\end{equation}
The biharmonic nonlinear Schr\"{o}dinger equation \eqref{GrindEQ__1_2_} was introduced by Karpman \cite{K96} and Karpman-Shagalov \cite{KS97} to take into account the role of small fourth-order dispersion terms in the propagation of intense laser beams in a bulk medium with Kerr nonlinearity and it has attracted a lot of interest during the last two decades. See, for example, \cite{D18B,D18I,D21,GC07,LZ21,MZ16,P07,PX13} and the references therein.
The local and global well-posedness as well as scattering and blow-up in the energy space $H^{2}$ have been widely studied.
See \cite{D21,MZ16,P07,PX13} for example.
On the other hand, the local and global well-posedness in the fractional Sobolev spaces $H^s$ for the 4NLS equation \eqref{GrindEQ__1_1_} have also been studied by several authors. See \cite{D18B,D18I,GC07,LZ21} for example.

The equation \eqref{GrindEQ__1_1_} has a counterpart for the Laplacian operator, namely, the inhomogeneous nonlinear Schr\"{o}dinger (INLS) equation
\begin{equation} \label{GrindEQ__1_3_}
iu_{t} +\Delta u=\lambda |x|^{-b} |u|^{\sigma} u.
\end{equation}
The INLS equation \eqref{GrindEQ__1_3_} arises in nonlinear optics for modeling the propagation of laser beam and it has been widely studied by many authors. See, for example, \cite{AT21,AK211,AK212,AK23, AKC22, C21,DK21,GM21, MMZ21} and the references therein.

The IBNLS equation \eqref{GrindEQ__1_1_} is invariant under scaling $u_{\alpha}(t,x)=\alpha^{\frac{4-b}{\sigma}}u(\alpha^{4}t,\alpha x ),~\alpha >0$.
An easy computation shows that
$$
\left\|u_{\alpha}(t)\right\|_{\dot{H}^{s}}=\alpha^{s+\frac{4-b}{\sigma}-\frac{d}{2}}\left\|u(t)\right\|_{\dot{H}^{s}}.
$$
We thus define the critical Sobolev index
\begin{equation} \label{GrindEQ__1_4_}
s_{c}:=\frac{d}{2}-\frac{4-b}{\sigma}.
\end{equation}
Putting
\begin{equation} \label{GrindEQ__1_5_}
\sigma_{c}(s):=
\left\{\begin{array}{cl}
{\frac{8-2b}{d-2s},} ~&{{\rm if}~s<\frac{d}{2},}\\
{\infty,}~&{{\rm if}~s\ge \frac{d}{2},}
\end{array}\right.
\end{equation}
we can easily see that $s>s_{c}$ is equivalent to $\sigma<\sigma_{c}(s)$. If $s<\frac{d}{2}$, $s=s_{c}$ is equivalent to $\sigma=\sigma_{c}(s)$.
For initial data $u_{0}\in H^{s}(\mathbb R^{d})$, we say that the Cauchy problem \eqref{GrindEQ__1_1_} is $H^{s}$-critical (for short, critical) if $0\le s<\frac{d}{2}$ and $\sigma=\sigma_{c}(s)$.
If $s\ge 0$ and $\sigma<\sigma_{c}(s)$, the problem \eqref{GrindEQ__1_1_} is said to be  $H^{s}$-subcritical (for short, subcritical).
Especially, if $\sigma =\frac{8}{d-2s}$, the problem is known as $L^{2}$-critical or mass-critical.
If $\sigma =\frac{8-2b}{d-4}$ with $d\ge 5$, it is called $H^{2}$-critical or energy-critical.
Throughout the paper, a pair $(p,q)$ is said to be  admissible, for short $(p,q)\in A$, if
\begin{equation}\label{GrindEQ__1_6_}
\frac{2}{p}+\frac{d}{q}\le\frac{d}{2},~p\in [2,\infty],~q\in [2,\infty).
\end{equation}
Especially, we say that a pair $(p,q)$ is biharmonic Schr\"{o}dinger admissible (for short $(p,q)\in B$), if $(p,q)\in A$ with $\frac{4}{p}+\frac{d}{q}=\frac{d}{2}$. We also say that a pair $(p,q)$ is non-endpoint biharmonic Schr\"{o}dinger admissible (for short $(p,q)\in B_{0}$), if $(p,q)\in B$ with $p>2$. Note that $(p,q)\in B_{0}$ if, and only if, $\frac{4}{p}+\frac{d}{q}=\frac{d}{2}$ and
\begin{equation} \nonumber
\left\{\begin{array}{ll}
{2\le q< \frac{2d}{d-2}},~&{{\rm if}~d\ge 5,} \\
{2\le q<\infty} ,~&{{\rm if}~d\le 4.}
\end{array}\right.
\end{equation}
If $(p,q)\in A$ and $\frac{2}{p}+\frac{d}{q}=\frac{d}{2}$, then a pair $(p,q)$ is said to be Schr\"{o}dinger admissible (for short $(p,q)\in S$).
We also denote for $(p,q)\in [1,\infty]^{2}$,
\begin{equation}\label{GrindEQ__1_7_}
\gamma_{p,q}=\frac{d}{2}-\frac{4}{p}-\frac{d}{q}.
\end{equation}

The IBNLS equation \eqref{GrindEQ__1_1_} has attracted a lot of interest in recent years. See, for example, \cite{AKR22,ARK23,CG21, CGP22, GP20, GP22, LZ212, S21,S22,SG22} and the references therein.
Guzm\'{a}n-Pastor \cite{GP20} proved that \eqref{GrindEQ__1_1_} is locally well-posed in $L^{2}$, if $d\in \mathbb N$, $0<b<\min\left\{4,d\right\}$ and $0<\sigma<\sigma_{c}(0)$. They also established the local well-posedness in $H^{2}$ for $d\ge 3$, $0<b<\min\{4,\frac{d}{2}\}$, $\max\{0,\frac{2-2b}{d}\}<\sigma<\sigma_{c}(2)$.
Afterwards, Cardoso-Guzm\'{a}n-Pastor \cite{CGP22} established the local and global well-posedness in $\dot{H}^{s_{c}}\cap \dot{H}^{2}$ with $d\ge 5$, $0<s_{c}<2$, $0<b<\min\{4,\frac{d}{2}\}$ and $\max\{1,\frac{8-2b}{d}\}<\sigma< \frac{8-2b}{d-4}$.
Recently, Liu-Zhang \cite{LZ212} established the local well-posedness in $H^{s}$ with $0<s\le 2$ by using the Besov space theory.
More precisely, they proved that the IBNLS equation \eqref{GrindEQ__1_1_} is locally well-posed in $H^{s}$ if $d\in \mathbb N$, $0<s\le 2$, $0<\sigma<\sigma_{c}(s)$ and $0<b<\min\{4,\frac{d}{2}\}$. See Theorem 1.5 of \cite{LZ212} for details.
This result about the local well-posedness of \eqref{GrindEQ__1_1_} improves the one of \cite{GP20} by not only extending the validity of $d$ and $s$ but also removing the lower bound $\sigma>\frac{2-2b}{d}$.
This local well-posedness result is directly applied to obtain the global well-posedness result in $H^{2}$ for the mass-subcritical case $0<\sigma< \frac{8-2b}{d}$ and mass-critical case  $\sigma=\frac{8-2b}{d}$. See Corollary 1.10 of \cite{LZ212}. The global well-posedness and scattering in $H^{2}$ for the intercritical case $\frac{8-2b}{d}<\sigma<\sigma_{c}(2)$ have also been widely studied. See \cite{CG21, CGP22, GP20, GP22, S21,S22,SG22} for example.

Very recently, the authors in \cite{AKR22,ARK23} investigated the local and global well-posedness in the fractional Sobolev spaces $H^s$ with $0\le s <\min \{2+\frac{d}{2},\frac{3}{2}d\}$ for the IBNLS equation \eqref{GrindEQ__1_1_}. More precisely, they obtained the following local well-posedness result.
\begin{theorem}[\cite{AKR22}]\label{thm 1.1.}
Let $d\in \mathbb N$, $0\le s <\min \{2+\frac{d}{2},\frac{3}{2}d\}$, $0<b<\min\{4,d,\frac{3}{2}d-s,\frac{d}{2}+2-s\}$  and $0<\sigma<\sigma_{c}(s)$.
If $\sigma$ is not an even integer, assume that
\footnote[1]{Given $a\in \mathbb R$, $a^{+}$ denotes the fixed number slightly larger than $a$ ($a^{+}=a+\varepsilon$ with $\varepsilon>0$ small enough).}
\begin{equation} \label{GrindEQ__1_8_}
\sigma\ge \sigma_{\star}(s):=
\left\{\begin{array}{ll}
{0,}~&{{\rm if}~{d\in \mathbb N ~{\rm and}~s\le 1,}}\\
{\left[ s-\frac{d}{2}\right]},~&{{\rm if}~d=1,2 ~{\rm and}~ s\ge 1+\frac{d}{2}},\\
{\lceil s\rceil-2},~&{{\rm if}~d\ge 3~{\rm and}~s>2},\\
{\left(\frac{2s-2b-2}{d}\right)^{+}},~&{\rm otherwise.}
\end{array}\right.
\end{equation}
 Then for any $u_{0} \in H^{s}$, there exist $T_{\max}=T_{\max } \left(\left\|u_{0} \right\| _{H^{s} } \right)>0$, $T_{\min}=T_{\min } \left(\left\|u_{0} \right\| _{H^{s} } \right)>0$ and a unique, maximal solution of \eqref{GrindEQ__1_1_} satisfying
\begin{equation} \label{GrindEQ__1_9_}
u\in C\left(\left(-T_{\min},T_{\max} \right),H^{s} \right)\cap L_{\rm loc}^{p} \left(\left(-T_{\min },T_{\max} \right),H_{q}^{s} \right),
\end{equation}
for any $(p,q)\in B$. If $s>0$, then the solution of \eqref{GrindEQ__1_1_} depends continuously on the initial data $u_{0} $ in the following sense. There exists $0<T<T_{\max } ,\, T_{\min } $ such that if $u_{0,n} \to u_{0} $ in $H^{s}$ and if $u_{n} $ denotes the solution of \eqref{GrindEQ__1_1_} with the initial data $u_{0,n}$, then $0<T<T_{\max } \left(u_{0,n} \right),T_{\min } \left(u_{0,n}\right)$ for all sufficiently large $n$ and $u_{n} \to u$ in $L^{p} \left([-T,T],H_{q}^{s-\varepsilon } \right)$ as $n\to \infty $ for all $\varepsilon >0$ and all $(p,q)\in B$. In particular, $u_{n} \to u$ in $C\left([-T,T],H^{s-\varepsilon } \right)$ for all $\varepsilon >0$.
\end{theorem}

But the continuous dependence of the solution on the initial data in the sense of Theorem \ref{thm 1.1.} is weaker than what would be ``standard", i.e. $\varepsilon=0$. (see e.g. \cite{C03}).
As in the study of the classic NLS equation (cf. \cite{CFH11,DYC13}), INLS equation (cf. \cite{AKC22}) and the classic biharmonic NLS equation (cf. \cite{D18B,D18I,LZ21}), we wonder if the solution of the Cauchy problem for the IBNLS equation \eqref{GrindEQ__1_1_} depends on the initial data in the standard sense in $H^s$, i.e. in the sense that the local solution flow is continuous $H^s\to H^s$.

This paper aims to solve this problem. The continuous dependence of the Cauchy problem \eqref{GrindEQ__1_1_} in the standard sense in $H^s$  was already investigated by \cite{LZ212}. However, they only treated the case $s\le 2$ and $b<\frac{d}{2}$. In this paper, we study the continuous dependence of the Cauchy problem for the IBNLS equation \eqref{GrindEQ__1_1_} in the standard sense in $H^s$ with $0<s<\min \{2+\frac{d}{2},\frac{3}{2}d\}$.

The main ingredients of this paper are the estimates of $f(u)-f(v)$ in the fractional Sobolev spaces established in Section \ref{sec 3.}, where $f(u)$ is a nonlinear function which behaves like $\lambda|u|^{\sigma}u$ with $\lambda\in \mathbb R$. These estimates improve the similar results of \cite{AKC22,D18I} (see Lemmas 3.1--3.4, 3.8--3.11 of \cite{AKC22} and Corollary 3.5 of \cite{D18I}) and play a crucial role in proving the standard continuous dependence of the Cauchy problem for the IBNLS equation \eqref{GrindEQ__1_1_}. They can also be applied to improve the continuous dependence results on the similar equations such as the classic biharmonic NLS equation and the fractional NLS equation. In this paper, we only focus on the IBNLS equation \eqref{GrindEQ__1_1_}. Our main result is as follows.
\begin{theorem}\label{thm 1.2.}
Let $d\in \mathbb N$, $0< s <\min \{2+\frac{d}{2},\frac{3}{2}d\}$, $0<b<\min\{4,d,\frac{3}{2}d-s,\frac{d}{2}+2-s\}$  and $0<\sigma<\sigma_{c}(s)$.
If $\sigma$ is not an even integer, assume that
\begin{equation} \label{GrindEQ__1_10_}
\sigma\ge \tilde{\sigma}_{\star}(s):=
\left\{\begin{array}{ll}
{\left[ s-\frac{d}{2}\right]},~&{{\rm if}~d=1,2 ~{\rm and}~ s\ge 1+\frac{d}{2}},\\
{\lceil s\rceil-2},~&{{\rm if}~d\ge 3~{\rm and}~s>2},\\
{1},~&{\rm otherwise.}
\end{array}\right.
\end{equation}
 Then for any $u_{0} \in H^{s}$, then the corresponding solution $u$ of the IBNLS equation \eqref{GrindEQ__1_1_} depends continuously on the initial data $u_{0}$ in the following sense. For any interval $[-S,T]\subset (T_{\min }(u_0), T_{\max }(u_0)$, and every $(p,q)\in B$, if $u_{0,n} \to u_{0}$ in $H^{s}$ and if $u_{n} $ denotes the solution of \eqref{GrindEQ__1_1_} with the initial data $u_{0,n}$, then $u_{n} \to u$ in $L^{p} \left([-S,T],H_{q}^{s}\right)$ as $n\to \infty $. In particular, $u_{n} \to u$ in $C\left([-S,T],H^{s} \right)$. In addition, if $\sigma$ is an even integer, or if not we assume $\sigma\ge \hat{\sigma}_{\star}(s)$, then the dependence is locally Lipschitz, where
\begin{equation} \label{GrindEQ__1_11_}
\hat{\sigma}_{\star}(s):=
\left\{\begin{array}{ll}
{\left[s-\frac{d}{2}\right]+1},~&{{\rm if}~d=1,2 ~{\rm and}~ s\ge 1+\frac{d}{2}},\\
{\lceil s\rceil-1},~&{{\rm if}~d\ge 3~{\rm and}~s>2},\\
{1},~&{\rm otherwise.}
\end{array}\right.
\end{equation}
\end{theorem}

Let us compare Theorem \ref{thm 1.2.} with the result of \cite{LZ212}.

First, Liu-Zhang \cite{LZ212} only considered the case $s\le 2$. As an immediate consequence of Theorem \ref{thm 1.2.}, we have the following new continuous dependence result in $H^{s}(\mathbb R^{d})$ with $d\ge 2$ and $2<s<2+\frac{d}{2}$.
\begin{corollary}\label{cor 1.3.}
Let $d\ge 2$, $2<s<2+\frac{d}{2}$, $0<b<\min\{4,\frac{d}{2}+2-s\}$ and $0<\sigma<\sigma_{c}(s)$.
If $\sigma$ is not an even integer, assume further that
\begin{equation} \label{GrindEQ__1_12_}
\left\{\begin{array}{ll}
{\sigma\ge\left[ s-\frac{d}{2}\right]},~&{{\rm if}~d=2,}\\
{\sigma\ge\lceil s\rceil-2},~&{{\rm if}~d\ge 3.}\\
\end{array}\right.
\end{equation}
Then for any $u_{0} \in H^{s}$, then the corresponding solution of the IBNLS equation \eqref{GrindEQ__1_1_} depends continuously on the initial data $u_{0}$ in the sense of Theorem \ref{thm 1.2.}. In addition, if $\sigma$ is an even integer, or if not we assume
\begin{equation} \label{GrindEQ__1_13_}
\left\{\begin{array}{ll}
{\sigma\ge\left[ s-\frac{d}{2}\right]+1},~&{{\rm if}~d=2,}\\
{\sigma\ge\lceil s\rceil-1},~&{{\rm if}~d\ge 3.}\\
\end{array}\right.
\end{equation}
 then the dependence is locally Lipschitz.
\end{corollary}

When $0<s<2$, the continuous dependence result of \cite{LZ212} only covers the case $b<\min\{4,\frac{d}{2}\}$. As an immediate consequence of Theorem \ref{thm 1.2.}, we also have the following new result, which covers the case $\frac{d}{2}\le b< 4$ with $d<8$.

\begin{corollary}\label{cor 1.4.}
Let $d<8$, $0<s<\min\{2,d\}$, $\frac{d}{2}\le b<\min\{4,d,\frac{3}{2}d-s,\frac{d}{2}+2-s\}$ and $1\le \sigma<\sigma_{c}(s)$.
Then for any $u_{0} \in H^{s}$, then the corresponding solution of the IBNLS equation \eqref{GrindEQ__1_1_} depends continuously on the initial data $u_{0}$ in the standard sense in $H^s$.
\end{corollary}

Combining Corollaries \ref{cor 1.3.} and \ref{cor 1.4.} directly with the continuous dependence result of \cite{LZ212}, we obtain the following standard continuous dependence results for the IBNLS equation \eqref{GrindEQ__1_1_}.
\begin{corollary}\label{cor 1.5.}
Let $d\ge 2$, $0<s<2+\frac{d}{2}$, $0<b<\min\{4,d,\frac{d}{2}+2-s\}$ and $0<\sigma<\sigma_{c}(s)$.
If $\sigma$ is not an even integer, assume further that
\begin{equation} \label{GrindEQ__1_14_}
\left\{\begin{array}{ll}
{\sigma\ge\left[ s-\frac{d}{2}\right]},~&{{\rm if}~d=2~{\rm and}~s>2,}\\
{\sigma\ge\lceil s\rceil-2},~&{{\rm if}~d\ge 3~{\rm and}~s>2,}\\
{\sigma\ge1},~&{{\rm if}~s< 2~{\rm and}~b\ge \frac{d}{2},}\\
{\sigma>0},~&{{\rm otherwise}.}\\
\end{array}\right.
\end{equation}
Then for any $u_{0} \in H^{s}$, then the corresponding solution of the IBNLS equation \eqref{GrindEQ__1_1_} depends continuously on the initial data $u_{0}$ in the standard sense in $H^s$.
\end{corollary}

\begin{corollary}\label{cor 1.6.}
Let $d=1$, $0< s\le 2$, $0<b<\tilde{4}$ and $0<\sigma<\sigma_{c}(s)$, where
\begin{equation} \label{GrindEQ__1_15_}
\tilde{4}=
\left\{\begin{array}{ll}
{\min\{1,\frac{3}{2}-s\}},~&{{\rm if}~s<1,}\\
{\frac{1}{2}},~&{{\rm if}~s\ge 1.}\\
\end{array}\right.
\end{equation}
If $\sigma$ is not an even integer, assume further that
\begin{equation} \label{GrindEQ__1_16_}
\left\{\begin{array}{ll}
{\sigma\ge1},~&{{\rm if}~s<1~{\rm and}~b\ge\frac{1}{2},}\\
{\sigma>0},~&{{\rm otherwise}.}\\
\end{array}\right.
\end{equation}
Then for any $u_{0} \in H^{s}$, then the corresponding solution of the IBNLS equation \eqref{GrindEQ__1_1_} depends continuously on the initial data $u_{0}$ in the standard sense in $H^s$.
\end{corollary}

This paper is organized as follows. In Section \ref{sec 2.}, we recall some notation and give some preliminary results related to our problem. In Section \ref{sec 3.}, we establish the estimates of the term $f(u)-f(v)$ in the fractional Sobolev spaces, where $f(u)$ is a nonlinear function which behaves like $\lambda|u|^{\sigma}u$ with $\lambda\in \mathbb R$. In Section \ref{sec 4.}, we prove Theorem \ref{thm 1.2.}.

%%%%%%%%%%%%%%%%%%%%%%%%%%%%%%%%%%%%%%%%%%%%%%%%%%%%%%%%%%%%%%%%%%%%%%%%%%%%%
\section{Preliminaries}\label{sec 2.}
Let us introduce some notation used throughout the paper.
$C>0$ stands for a positive universal constant, which can be different at different places.
The notation $a\lesssim b$ means $a\le Cb$ for some constant $C>0$.
For $p\in \left[1,\infty \right]$, $p'$ denotes the dual number of $p$, i.e. $1/p+1/p'=1$.
For $s\in\mathbb R$, we denote by $[s]$ the largest integer which is less than or equal to $s$ and by $\left\lceil s\right\rceil $ the minimal integer which is larger than or equal to $s$.
As in \cite{WHHG11}, for $s\in \mathbb R$ and $1<p<\infty $, we denote by $H_{p}^{s} (\mathbb R^{d} )$ and $\dot{H}_{p}^{s} (\mathbb R^{d} )$ the nonhomogeneous Sobolev space and homogeneous Sobolev space, respectively. As usual, we abbreviate $H_{2}^{s} (\mathbb R^{d} )$ and $\dot{H}_{2}^{s} (\mathbb R^{d} )$ as $H^{s} (\mathbb R^{d} )$ and $\dot{H}^{s} (\mathbb R^{d} )$, respectively. Given two normed spaces $X$ and $Y$, $X\hookrightarrow Y$ means that $X$ is continuously embedded in $Y$, i.e. there exists a constant $C\left(>0\right)$ such that $\left\| f\right\| _{Y} \le C\left\| f\right\| _{X} $ for all $f\in X$. If there is no confusion, $\mathbb R^{d} $ will be omitted in various function spaces.

Next, we recall some useful facts which are used throughout the paper.
First of all, we recall some useful embeddings on Sobolev spaces. See \cite{WHHG11} for example.

\begin{lemma}\label{lem 2.1.}
Let $-\infty <s_{2} \le s_{1} <\infty $ and $1<p_{1} \le p_{2} <\infty $ with $s_{1} -\frac{d}{p_{1} } =s_{2} -\frac{d}{p_{2} } $. Then we have the following embeddings:
\[\dot{H}_{p_{1} }^{s_{1} } \hookrightarrow \dot{H}_{p_{2} }^{s_{2} } ,H_{p_{1} }^{s_{1} } \hookrightarrow H_{p_{2} }^{s_{2} } .\]
\end{lemma}

\begin{lemma}\label{lem 2.2.}
Let $-\infty <s<\infty $ and $1<p<\infty $. Then we have
\begin{enumerate}
\item $H_{p}^{s+\varepsilon } \hookrightarrow H_{p}^{s}~(\varepsilon >0)$,
\item $H_{p}^{s} =L^{p} \cap \dot{H}_{p}^{s}~(s>0)$.
\end{enumerate}
\end{lemma}

\begin{corollary}\label{cor 2.3.}
Let $-\infty <s_{2} \le s_{1} <\infty $ and $1<p_{1} \le p_{2} <\infty $ with $s_{1} -\frac{d}{p_{1} } \ge s_{2} -\frac{d}{p_{2} } $. Then we have $H_{p_{1} }^{s_{1} } \hookrightarrow H_{p_{2} }^{s_{2} } $.
\end{corollary}
\begin{proof} The result follows from Lemma \ref{lem 2.1.} and Item 1 of Lemma \ref{lem 2.2.}.
\end{proof}

Next, we recall the well-known fractional chain rule and product rule.

\begin{lemma}[Fractional Chain Rule, \cite{CW91, KPV93}]\label{lem 2.4.}
Suppose $G\in C^{1} (\mathbb C)$ and $s\in \left(0,\;1\right)$. Then for $1<r,\;r_{2} <\infty $, and $1<r_{1} \le \infty $ satisfying $\frac{1}{r} =\frac{1}{r_{1} } +\frac{1}{r_{2} } $,
\begin{equation} \nonumber
\left\| G(u)\right\| _{\dot{H}_{r}^{s} } \lesssim\left\| G'(u)\right\| _{r_{1} } \left\|u\right\|_{\dot{H}_{r_{2} }^{s} } .
\end{equation}
\end{lemma}

\begin{lemma}[Fractional Product Rule, \cite{CW91}]\label{lem 2.5.}
Let $s\ge 0$, $1<r,r_{2},p_{1}<\infty$ and $1<r_{1},p_{2}\le\infty$. Assume that
\[\frac{1}{r} =\frac{1}{r_{i} } +\frac{1}{p_{i} }~(i=1,2).\]
Then we have
\begin{equation}
\left\| fg\right\| _{\dot{H}_{r}^{s} } \lesssim \left\| f\right\| _{r_{1} } \left\| g\right\| _{\dot{H}_{p_{1} }^{s} } +\left\| f\right\| _{\dot{H}_{r_{2} }^{s} } \left\| g\right\| _{p_{2} } .
\end{equation}
\end{lemma}

Using Lemma \ref{lem 2.5.}, H\"{o}lder inequality and the induction, we can get the following result.

\begin{corollary}\label{cor 2.6.}
Let $s\ge 0$, $q\in \mathbb N$. Let ${1}<r,\;r_{k}^{i} <\infty $ for $1\le i,\;k\le q$. Assume that
\[\frac{1}{r} =\sum _{i=1}^{k}\frac{1}{r_{k}^{i} }  , \]
for any $1\le k\le q$. Then we have
\begin{equation} \nonumber
\left\| \prod _{i=1}^{k}f_{i}  \right\| _{\dot{H}_{r}^{s} } \lesssim\sum _{k=1}^{k}(\left\| f_{k} \right\| _{\dot{H}_{r_{k}^{k} }^{s} } \prod _{i\in I_{k} }\left\| f_{i} \right\| _{r_{k}^{i} }),
\end{equation}
where $I_{k} =\left\{i\in \mathbb N:\;1\le i\le q,\;i\ne k\right\}$.
\end{corollary}

\begin{lemma}[\cite{AK211}]\label{lem 2.7.}
Let $s>0$, $1<p<\infty $ and $v=s-[s]$. Then $\sum _{|\alpha |=[s]}\left\| D^{\alpha } f\right\| _{\dot{H}_{p}^{v} }  $ is an equivalent norm on $\dot{H}_{p}^{s}$.
\end{lemma}

We also recall the interpolation inequality in Sobolev spaces. See, for example, Proposition 1.21 of \cite{WHHG11}.
\begin{lemma}[Convexity H\"{o}lder inequality]\label{lem 2.8.}
 Let $1<p,\;p_{i} <\infty $, $0\le \theta _{i} \le 1$, $s_{i} ,\;s\in \mathbb R$, ${\rm (}i=1,\;\ldots ,\;N{\rm )}$, $\sum _{i=1}^{N}\theta _{i}  =1$, $s=\sum _{i=1}^{N}\theta _{i}  s_{i} $, $1/p =\sum _{i=1}^{N}{\theta _{i}/p_{i} }$. Then we have $\bigcap _{i=1}^{N}\dot{H}_{p_{i} }^{s_{i} }  \subset \dot{H}_{p}^{s} $ and for any $v\in \bigcap _{i=1}^{N}\dot{H}_{p_{i} }^{s_{i} } $,
\begin{equation} \nonumber
\left\| v\right\| _{\dot{H}_{p}^{s} } \le \prod _{i=1}^{N}\left\| v\right\| _{\dot{H}_{p_{i} }^{s_{i} } }^{\theta _{i} }  .
\end{equation}
\end{lemma}

We end this section with recalling the Strichartz estimates for the fourth-order Schr\"{o}dinger equation. See, for example, \cite{D18I} and the reference therein.
\begin{lemma}[Strichartz estimates]\label{lem 2.9.}
Let $\gamma\in \mathbb R$ and $u$ be the solution to the linear fourth-order Schr\"{o}dinger equation, namely
\begin{equation}\nonumber
u(t)=e^{it\Delta^{2}}u_{0}+i\int_{0}^{t}{e^{i(t-s)\Delta^{2}}F(s)ds,}
\end{equation}
for some data $u_{0}$ and $F$. Then for all $(p,q)$ and $(a,b)$ admissible,
$$
\left\|u\right\|_{L^{p}(I,\dot{H}_{q}^{\gamma})}\lesssim  \left\|u_{0}\right\|_{\dot{H}^{\gamma+\gamma_{p,q}}}+ \left\|F\right\|_{L^{a'}(I,\dot{H}_{b'}^{\gamma+\gamma_{p,q}-\gamma_{a',b'}-4})},
$$
where $\gamma_{p,q}$ and $\gamma_{a',b'}$ are as in \eqref{GrindEQ__1_9_}.
\end{lemma}
%%%%%%%%%%%%%%%%%%%%%%%%%%%%%%%%%%%%%%%%%%%%%%%%%%%%%%%%%%%%%%%%%%%%%%%%%%%%%
\section{Estimates of $f(u)-f(v)$}\label{sec 3.}

In this section, we establish the estimates of the term $f(u)-f(v)$, which generalize Corollary 3.5 of \cite{D18I} and Lemmas 3.1--3.4, 3.8--3.11 of \cite{AKC22}.

\begin{lemma}\label{lem 3.1.}
Let $0<s<1$, $\sigma \ge 1$ and $1<p,r<\infty$, $1<q\le \infty$ satisfying $\frac{1}{p}=\frac{1}{r}+\frac{\sigma}{q}$. Assume that $f\in C^{2} \left(\mathbb C\to \mathbb C\right)$ satisfies
\begin{equation} \label{GrindEQ__3_1_}
|f^{(k)} (u)|\lesssim|u|^{\sigma +1-k} ,
\end{equation}
for any $0\le k\le 2$ and $u\in \mathbb C$. Then we have
\begin{eqnarray}\begin{split}\nonumber
\left\| f(u)-f(v)\right\| _{\dot{H}_{p}^{s} } \lesssim&(\left\| u\right\| _{L^q}^{\sigma } +\left\| v\right\| _{L^q}^{\sigma })\left\| u-v\right\| _{\dot{H}_{r}^{s} } \\
&+(\left\| u\right\| _{L^q}^{\sigma-1} +\left\| v\right\| _{L^q}^{\sigma-1})
(\left\| u\right\| _{\dot{H}_{r}^{s}}+\left\| v\right\| _{\dot{H}_{r}^{s}})\left\| u-v\right\| _{L^{q}}.
\end{split}\end{eqnarray}
\end{lemma}
\begin{proof}
Without loss of generality and for simplicity, we assume that $f$ is a function of a real variable. We have
\[f(u)-f(v)=\left(u-v\right)\int _{0}^{1}f'\left(v+t\left(u-v\right)\right)dt .\]
Putting
\begin{equation} \label{GrindEQ__3_2_}
\frac{1}{p_{1} } =\frac{\sigma }{q} ,~\frac{1}{p_{2} } =\frac{\sigma-1}{q} +\frac{1}{r},
\end{equation}
it follows from \eqref{GrindEQ__3_1_} and Lemma \ref{lem 2.5.} (fractional product rule) that
\begin{equation}\nonumber
\left\| f(u)-f(v)\right\| _{\dot{H}_{p}^{s} }=\left\| \left(u-v\right)\int _{0}^{1}f'\left(v+t\left(u-v\right)\right)dt \right\| _{\dot{H}_{p}^{s} }\lesssim I_{1} +I_{2},
\end{equation}
where

$$
I_{1}=\left\| \int _{0}^{1}f'\left(v+t\left(u-v\right)\right)dt \right\| _{L^{p_1}} \left\| u-v\right\| _{\dot{H}_{r}^{s} },
$$

$$
I_{2}=\left\| \int _{0}^{1}f'\left(v+t\left(u-v\right)\right)dt \right\| _{\dot{H}_{p_{2} }^{s} } \left\| u-v\right\| _{L^q}.
$$
First, we estimate $I_{1}$. It follows from \eqref{GrindEQ__3_1_} that
\begin{equation} \label{GrindEQ__3_3_}
|f'\left(v+t\left(u-v\right)\right)|\lesssim{\mathop{\max }\limits_{t\in \left[0,\, 1\right]}} \left|v+t\left(u-v\right)\right|^{\sigma } \lesssim |u|^{\sigma } +|v|^{\sigma },
\end{equation}
for any $0\le t \le 1$. Using \eqref{GrindEQ__3_2_} and \eqref{GrindEQ__3_3_}, we have
\begin{eqnarray}\begin{split} \label{GrindEQ__3_4_}
I_{1} &\le \left\| u-v\right\| _{\dot{H}_{r}^{s} }\int _{0}^{1}\left\| f'\left(v+t\left(u-v\right)\right)\right\| _{L^{p_{1}} } dt \\
&\lesssim \left\| |u|^{\sigma }+|v|^{\sigma } \right\| _{p_{1} }  \left\| u-v\right\| _{\dot{H}_{r}^{s} }\le (\left\| u\right\| _{L^q}^{\sigma } +\left\| v\right\| _{L^q}^{\sigma })\left\| u-v\right\| _{\dot{H}_{r}^{s} }.
\end{split}\end{eqnarray}
Next, we estimate $I_{2}$. We have
\begin{equation} \label{GrindEQ__3_5_}
\left\| \int _{0}^{1}f'\left(v+t\left(u-v\right)\right)dt \right\| _{\dot{H}_{p_{2} }^{s} } \le \int _{0}^{1}\left\| f'\left(v+t\left(u-v\right)\right)\right\| _{\dot{H}_{p_{2} }^{s} } dt .
\end{equation}
It also follows from Lemma \ref{lem 2.4.} (fractional chain rule) and \eqref{GrindEQ__3_1_} that
\begin{eqnarray}\begin{split} \label{GrindEQ__3_6_}
\left\| f'\left(v+t\left(u-v\right)\right)\right\| _{\dot{H}_{p_{2} }^{s} } &\lesssim\left\| f''\left(v+t\left(u-v\right)\right)\right\| _{L^{p_{3}}} \left\| v+t\left(u-v\right)\right\| _{\dot{H}_{r}^{s} }.
\end{split}\end{eqnarray}
where $\frac{1}{p_{3} } =\frac{\sigma -1}{q}$ if $\sigma >1$, and $p_{3} =\infty $ if $\sigma =1$. We can also see that
\begin{equation} \label{GrindEQ__3_7_}
|f''\left(v+t\left(u-v\right)\right)|\lesssim{\mathop{\max }\limits_{t\in \left[0,\, 1\right]}} \left|v+t\left(u-v\right)\right|^{\sigma-1} \lesssim |u|^{\sigma-1} +|v|^{\sigma-1},
\end{equation}
and
\begin{equation} \label{GrindEQ__3_8_}
\left\| v+t\left(u-v\right)\right\| _{\dot{H}_{r}^{s} } \le {\mathop{\max }\limits_{t\in \left[0,\, 1\right]}} \left\| I^{s} v+t\left(I^{s} u-I^{s} v\right)\right\| _{L^{r} } \lesssim\left\| u\right\| _{\dot{H}_{r}^{s} } +\left\| v\right\| _{\dot{H}_{r}^{s} },
\end{equation}
for any $0\le t \le 1$.
\eqref{GrindEQ__3_5_}--\eqref{GrindEQ__3_8_} and \eqref{GrindEQ__3_2_} imply that
\begin{eqnarray}\begin{split} \label{GrindEQ__3_9_}
I_{2}&=\left\| \int _{0}^{1}f'\left(v+t\left(u-v\right)\right)dt \right\| _{\dot{H}_{p_{2} }^{s} } \left\| u-v\right\| _{L^q} \\
&\lesssim
(\left\| u\right\| _{L^q}^{\sigma-1} +\left\| v\right\| _{L^q}^{\sigma-1})
(\left\| u\right\| _{\dot{H}_{r}^{s}}+\left\| v\right\| _{\dot{H}_{r}^{s}})\left\| u-v\right\| _{L^{q}} .
\end{split}\end{eqnarray}
In view of \eqref{GrindEQ__3_4_} and \eqref{GrindEQ__3_9_}, we get the desired result.
This completes the proof.
\end{proof}

\begin{lemma}\label{lem 3.2.}
Let $s\ge 1$, $\sigma \ge \left\lceil s\right\rceil -1$ and $1<p,r<\infty$, $1<q\le \infty$ satisfying $\frac{1}{p}=\frac{1}{r}+\frac{\sigma}{q}$.
Assume that $f\in C^{\left\lceil s\right\rceil } \left(\mathbb C\to \mathbb C\right)$ satisfies
\begin{equation} \label{GrindEQ__3_10_}
|f^{(k)} (u)|\lesssim|u|^{\sigma +1-k} ,
\end{equation}
for any $0\le k\le \left\lceil s\right\rceil $ and $u\in \mathbb C$. Assume further
\begin{equation} \label{GrindEQ__3_11_}
|f^{\left(\left\lceil s\right\rceil\right)} (u)-f^{\left(\left\lceil s\right\rceil \right)} (v)|
\lesssim\left|u-v\right|^{\min \{ \sigma -\left\lceil s\right\rceil +1,\;1\} } \left(|u|+|v|\right)^{\max \{ 0,\;\sigma -\left\lceil s\right\rceil \} } ,
\end{equation}
for any $u,\;v\in \mathbb C$. Then we have
\begin{eqnarray}\begin{split}\label{GrindEQ__3_12_}
&\left\| f(u)-f(v)\right\| _{\dot{H}_{p}^{s} } \lesssim(\left\| u\right\| _{L^q}^{\sigma } +\left\| v\right\| _{L^q}^{\sigma })\left\| u-v\right\| _{\dot{H}_{r}^{s} } \\
&~~~~~~+(\left\| u\right\| _{L^q}^{\sigma-1} +\left\| v\right\| _{L^q}^{\sigma-1})
(\left\| u\right\| _{\dot{H}_{r}^{s}}+\left\| v\right\| _{\dot{H}_{r}^{s}})\left\| u-v\right\| _{L^{q}}\\
&~~~~~~+(\left\| u\right\| _{L^q}^{\max \{ \left\lceil s\right\rceil-1 ,\;\sigma-1\} } +\left\| v\right\| _{L^q}^{\max \{ \left\lceil s\right\rceil-1,\;\sigma-1 \} } )\left\|u\right\| _{\dot{H}_{r}^{s}}\left\| u-v\right\| _{L^q}^{\min \{ \sigma -\left\lceil s\right\rceil +1,\;1\} }.
\end{split}\end{eqnarray}
\end{lemma}
\begin{proof}
If follows from Lemma \ref{lem 2.7.} that
\begin{equation} \label{GrindEQ__3_13_}
\left\| f(u)-f(v)\right\| _{\dot{H}_{p}^{s} } \lesssim\sum _{|\alpha |=\left[s\right]}\left\| D^{\alpha } f(u)-D^{\alpha } f(v)\right\| _{\dot{H}_{p}^{v} }  ,
\end{equation}
where $v=s-\left[s\right]$. Without loss of generality and for simplicity, we assume that $f$ is a function of a real variable. It follows from the Leibniz rule of derivatives that
\begin{equation} \label{GrindEQ__3_14_}
D^{\alpha } f(u)=\sum _{k=1}^{|\alpha |}\sum _{\Lambda _{\alpha }^{k} }C_{\alpha ,\;k} f^{(k)} (u)\prod _{i=1}^{k}D^{\alpha _{i} } u
\end{equation}
where $\Lambda _{\alpha }^{k} =\left(\alpha _{1} +\cdots +\alpha _{k} =\alpha ,\;|\alpha _{i}|\ge 1\right)$.
Hence, we have
\begin{eqnarray}\begin{split} \label{GrindEQ__3_15_}
\left\| D^{\alpha } f(u)-D^{\alpha } f(v)\right\| _{\dot{H}_{p}^{v} } &=\left\| \sum _{k=1}^{|\alpha |}\sum _{\Lambda _{\alpha }^{k} }C_{\alpha ,\;k} \left(f^{(k)} (u)\prod _{i=1}^{k}D^{\alpha _{i} } u -f^{(k)} (v)\prod _{i=1}^{k}D^{\alpha _{i} } v \right)  \right\| _{\dot{H}_{p}^{v} } \\
&\le \sum _{k=1}^{|\alpha |}\sum _{\Lambda _{\alpha }^{k} }C_{\alpha ,\;k} \left\| f^{(k)} (u)\prod _{i=1}^{k}D^{\alpha _{i} } u -f^{(k)} (v)\prod _{i=1}^{k}D^{\alpha _{i} } v \right\| _{\dot{H}_{p}^{v} }    \\
&\le \sum _{k=1}^{|\alpha |}\sum _{\Lambda _{\alpha }^{k} }C_{\alpha ,\;k} (\left\| I_{k} \right\| _{\dot{H}_{p}^{v} } +\left\| II_{k} \right\| _{\dot{H}_{p}^{v} }) ,
\end{split}\end{eqnarray}
where
\begin{equation} \label{GrindEQ__3_16_}
I_{k} =\left(f^{(k)} (u)-f^{(k)} (v)\right)\prod _{i=1}^{k}D^{\alpha _{i} } u ,~II_{k} =f^{(k)} (v)\left(\prod _{i=1}^{k}D^{\alpha _{i} } u -\prod _{i=1}^{k}D^{\alpha _{i} } v \right).
\end{equation}

We divide the study in two cases: $s\in \mathbb N$ and $s\notin \mathbb N$.

\textbf{Case 1.} First, we consider the case $s\in \mathbb N$, i.e. $v=0$.

\emph{Step 1.1}. First, we estimate $\left\| I_{k} \right\| _{p} $, where $I_{k}$ is given in \eqref{GrindEQ__3_16_}.
Using \eqref{GrindEQ__3_10_}, H\"{o}lder inequality, Lemmas \ref{lem 2.7.} and \ref{lem 2.8.}, we have
\begin{eqnarray}\begin{split}\label{GrindEQ__3_17_}
I_{k}& \le \left\| f^{(k)} (u)-f^{(k)} (v)\right\| _{L^{p_{1_k}}} \prod_{i=1}^{k}{\left\|D^{\alpha _{i} } u\right\|_{L^{\alpha_{k_i}}}}\\
&\lesssim \left\| f^{(k)} (u)-f^{(k)} (v)\right\| _{L^{p_{1_k}}}
\prod_{i=1}^{k} (\left\|u\right\| _{\dot{H}_{r}^{s}}^{\frac{|\alpha_{i}|}{s}}
\left\|u\right\| _{L^{q}}^{1-\frac{|\alpha_{i}|}{s}})\\
& \lesssim \left\|f^{(k)} (u)-f^{(k)} (v)\right\| _{L^{p_{1_k}}}\left\|u\right\| _{\dot{H}_{r}^{s}} \left\|u\right\| _{L^{q}}^{k-1},
\end{split}\end{eqnarray}
where
\begin{equation} \label{GrindEQ__3_18_}
\frac{1}{p_{1_k}}:=\frac{\sigma+1-k}{q},~\frac{1}{\alpha_{k_i}}:=\frac{|\alpha_{i}|}{s}\frac{1}{r}+\left(1-\frac{|\alpha_{i}|}{s}\right)\frac{1}{q}.
\end{equation}

$\cdot$ If $k=|\alpha |=\left\lceil s\right\rceil $, then it follows from \eqref{GrindEQ__3_11_} and H\"{o}lder inequality that
\begin{eqnarray}\begin{split} \label{GrindEQ__3_19_}
\left\| f^{\left(\left\lceil s\right\rceil \right)} (u)-f^{\left(\left\lceil s\right\rceil \right)} (v)\right\| _{L^{p_{1_k}}} &\lesssim\left\| \left|u-v\right|^{\min \{ \sigma -\left\lceil s\right\rceil +1,\;1\} } \left(|u|+|v|\right)^{\max \{ 0,\;\sigma -\left\lceil s\right\rceil \} } \right\| _{L^{p_{1_k}}}  \\
&\lesssim\left\| u-v\right\| _{L^q}^{\min \{ \sigma -\left\lceil s\right\rceil +1,\;1\} } (\left\| u\right\|_{L^q}+\left\|v\right\|_{L^q})^{\max \{ 0,\;\sigma -\left\lceil s\right\rceil \} }.
\end{split}\end{eqnarray}
In view of \eqref{GrindEQ__3_17_} and \eqref{GrindEQ__3_19_}, we have
\begin{equation} \label{GrindEQ__3_20_}
\left\| I_{k} \right\| _{L^p} \lesssim\left\| u-v\right\| _{L^q}^{\min \{ \sigma -\left\lceil s\right\rceil +1,\;1\} }(\left\| u\right\| _{L^q}^{\max \{ \left\lceil s\right\rceil-1 ,\;\sigma-1\} } +\left\| v\right\| _{L^q}^{\max \{ \left\lceil s\right\rceil-1,\;\sigma-1 \} } )\left\|u\right\| _{\dot{H}_{r}^{s}} .
\end{equation}

$\cdot$ If $k<|\alpha |=\left\lceil s\right\rceil $, then we have
\begin{eqnarray}\begin{split} \label{GrindEQ__3_21_}
\left\| f^{(k)} (u)-f^{(k)} (v)\right\| _{L^{p_{1_k}}}&=\left\| \left(u-v\right)\int _{0}^{1}f^{\left(k+1\right)} \left(v+t\left(u-v\right)\right)dt \right\| _{L^{p_{1_k}}}\\
&\lesssim \left\| \left(u-v\right)\int _{0}^{1}\left|v+t\left(u-v\right)\right|^{\sigma -k} dt \right\| _{L^{p_{1_k}}}\\
&\lesssim \left\| u-v\right\| _{L^{q} } (\left\| u\right\| _{L^{q}}^{\sigma -k} +\left\| v\right\| _{L^{q}}^{\sigma -k}).
\end{split}\end{eqnarray}
In view of \eqref{GrindEQ__3_17_} and \eqref{GrindEQ__3_21_}, we have
\begin{equation} \label{GrindEQ__3_22_}
\left\| I_{k} \right\| _{L^p} \lesssim(\left\| u\right\| _{L^{q}}^{\sigma-1} +\left\| v\right\| _{L^{q}}^{\sigma-1})\left\|u\right\|_{\dot{H}_r^s}\left\| u-v\right\| _{\dot{H}_{r}^{s} } .
\end{equation}
Hence we have
\begin{eqnarray}\begin{split} \label{GrindEQ__3_23_}
&\left\| I_{k} \right\| _{L^p} \lesssim
(\left\| u\right\| _{L^q}^{\sigma-1} +\left\| v\right\| _{L^q}^{\sigma-1})
(\left\| u\right\| _{\dot{H}_{r}^{s}}+\left\| v\right\| _{\dot{H}_{r}^{s}})\left\| u-v\right\| _{L^{q}}\\
&~~~~~~+(\left\| u\right\| _{L^q}^{\max \{ \left\lceil s\right\rceil-1 ,\;\sigma-1\} } +\left\| v\right\| _{L^q}^{\max \{ \left\lceil s\right\rceil-1,\;\sigma-1 \} } )\left\|u\right\| _{\dot{H}_{r}^{s}}\left\| u-v\right\| _{L^q}^{\min \{ \sigma -\left\lceil s\right\rceil +1,\;1\} }.
\end{split}\end{eqnarray}
for any $1\le k\le s=\left\lceil s\right\rceil $.

\emph{Step 1.2.} Next, we estimate $\left\| II_{k} \right\| _{L^p}$, where $II_{k}$ is given in \eqref{GrindEQ__3_16_}.
If $k=1$, it follows directly from \eqref{GrindEQ__3_10_} and H\"{o}lder inequality that
\begin{eqnarray}\begin{split}\label{GrindEQ__3_24_}
\left\|II_{1}\right\|_{L^p} &=\left\|f'(v)(D^{\alpha } u -D^{\alpha} v )\right\|_{L^p}
\lesssim \left\|v\right\|_{L^q}^{\sigma}\left\|D^{\alpha } (u -v) \right\|_{L^r}
\lesssim \left\|v\right\|_{L^q}^{\sigma}\left\|u - v \right\|_{\dot{H}_r^{s}}.
\end{split}\end{eqnarray}
Next, we consider the case $k\ge 2$. Noticing that
\footnote[1]{We assume that $\prod _{j=1}^{0}a_{i}  =\prod _{j=N+1}^{N}b_{i}  =0$.}
\begin{equation} \label{GrindEQ__3_25_}
\prod _{i=1}^{N}a_{i}  -\prod _{i=1}^{N}b_{i}  =\sum _{i=1}^{N}\prod _{j=1}^{i-1}a_{j}  \prod _{j=i+1}^{N}b_{j}  \left(a_{i} -b_{i} \right) ,
\end{equation}
we have
\begin{eqnarray}\begin{split} \label{GrindEQ__3_26_}
\left\| II_{k} \right\| _{L^p} &=\left\| f^{(k)} (v)\left(\prod _{i=1}^{k}D^{\alpha _{i} } u -\prod _{i=1}^{k}D^{\alpha _{i} } v \right)\right\| _{L^p}\\
&\le \sum _{i=1}^{k}\left\| f^{(k)} (v)\left(\prod _{j=1}^{i-1}D^{\alpha _{j} } u \prod _{j=i+1}^{k}D^{\alpha _{j} } v \left(D^{\alpha _{i} } u-D^{\alpha _{i} } v\right)\right)\right\| _{L^p}  .
\end{split}\end{eqnarray}
It follows from \eqref{GrindEQ__3_10_}, \eqref{GrindEQ__3_18_}, H\"{o}lder inequality and Lemma \ref{lem 2.7.} that
\begin{eqnarray}\begin{split} \label{GrindEQ__3_27_}
&\left\| f^{(k)} (v)\left(\prod _{j=1}^{i-1}D^{\alpha _{j} } u \prod _{j=i+1}^{k}D^{\alpha _{j} } v \left(D^{\alpha _{i} } u-D^{\alpha _{i} } v\right)\right)\right\| _{L^p}\\
&\lesssim \left\| u\right\| _{L^q}^{\sigma +1-k} \left\| u-v\right\| _{\dot{H}_{\alpha_{k_i} }^{|\alpha _{i}|} } \prod _{j=1}^{i-1}\left\| u\right\| _{\dot{H}_{\alpha_{k_j} }^{|\alpha _{j}|} }  \prod _{j=i+1}^{k}\left\| v\right\| _{\dot{H}_{\alpha_{k_j} }^{|\alpha _{j}|} }\\
&\lesssim \left\| u\right\| _{L^q}^{\sigma +1-k} \left\|u-v\right\| _{\dot{H}_{r}^{s}}^{\frac{|\alpha_{i}|}{s}}\left\|u-v\right\| _{L^{q}}^{1-\frac{|\alpha_{i}|}{s}}
\prod_{j=1}^{i-1} (\left\|u\right\| _{\dot{H}_{r}^{s}}^{\frac{|\alpha_{j}|}{s}}\left\|u\right\| _{L^{q}}^{1-\frac{|\alpha_{j}|}{s}})
\prod_{j=i+1}^{k} (\left\|v\right\| _{\dot{H}_{r}^{s}}^{\frac{|\alpha_{j}|}{s}}\left\|v\right\| _{L^{q}}^{1-\frac{|\alpha_{j}|}{s}})\\
&\lesssim \left[(\left\|u\right\| _{\dot{H}_{r}^{s}}+\left\|v\right\| _{\dot{H}_{r}^{s}})\left\|u-v\right\|_{L^q}\right]^{1-\frac{|\alpha_{i}|}{s}}
\left[\left\|u-v\right\| _{\dot{H}_{r}^{s}}(\left\|u\right\|_{L^q}+\left\|v\right\|_{L^q})\right]^{\frac{|\alpha_{i}|}{s}}
(\left\|u\right\|_{L^q}+\left\|v\right\|_{L^q})^{\sigma-1}\\
&\lesssim(\left\| u\right\| _{L^q}^{\sigma } +\left\| v\right\| _{L^q}^{\sigma })\left\| u-v\right\| _{\dot{H}_{r}^{s}}+
(\left\| u\right\| _{L^q}^{\sigma-1} +\left\| v\right\| _{L^q}^{\sigma-1})
(\left\| u\right\| _{\dot{H}_{r}^{s}}+\left\| v\right\| _{\dot{H}_{r}^{s}})\left\| u-v\right\| _{L^{q}},
\end{split}\end{eqnarray}
where the last inequality follows from the fact:
$$\beta_{1}^{\gamma_1}\beta_{2}^{\gamma_2}\le (\beta_{1}+\beta_{2})^{\gamma_1+\gamma_2},~ \textrm{for}~\beta_{1},\beta_{2},\gamma_1,\gamma_2>0.$$
In view of \eqref{GrindEQ__3_24_}, \eqref{GrindEQ__3_26_} and \eqref{GrindEQ__3_27_}, we have
\begin{equation} \label{GrindEQ__3_28_}
\left\| II_{k} \right\| _{L^p} \lesssim(\left\| u\right\| _{L^q}^{\sigma } +\left\| v\right\| _{L^q}^{\sigma })\left\| u-v\right\| _{\dot{H}_{r}^{s}}+
(\left\| u\right\| _{L^q}^{\sigma-1} +\left\| v\right\| _{L^q}^{\sigma-1})
(\left\| u\right\| _{\dot{H}_{r}^{s}}+\left\| v\right\| _{\dot{H}_{r}^{s}})\left\| u-v\right\| _{L^{q}},
\end{equation}
for any $1\le k\le s=\left\lceil s\right\rceil $.
\eqref{GrindEQ__3_12_} follows directly from \eqref{GrindEQ__3_23_} and \eqref{GrindEQ__3_28_}.
This completes the proof of \eqref{GrindEQ__3_17_} in the case $s\in \mathbb N$.

\textbf{Case 2.} We consider the case $s\notin \mathbb N$.

\emph{Step 2.1.} First, we estimate $\left\| I_{k} \right\| _{\dot{H}_{p}^{v} } $, where $I_{k}$ is given in \eqref{GrindEQ__3_16_}.

It follows from Lemma \ref{lem 2.5.} (fractional product rule) that
\begin{eqnarray}\begin{split} \label{GrindEQ__3_29_}
\left\| I_{k} \right\| _{\dot{H}_{p}^{v} } &\lesssim\left\| f^{(k)} (u)-f^{(k)} (v)\right\| _{p_{1_k} } \left\| \prod _{i=1}^{k}D^{\alpha _{i} } u \right\| _{\dot{H}_{a_{k} }^{v} } +\left\| f^{(k)} (u)-f^{(k)} (v)\right\| _{\dot{H}_{p_{2_k} }^{v} } \left\| \prod _{i=1}^{k}D^{\alpha _{i} } u \right\| _{b_{k} }\\
&\equiv I_{k_1} +I_{k_2} ,
\end{split}\end{eqnarray}
where
\begin{equation} \label{GrindEQ__3_30_}
\frac{1}{a_{k} }:=\frac{1}{r}+\frac{k-1}{q},~\frac{1}{p_{2_k} }: =\frac{\sigma-k}{q}+\frac{v}{rs}+\frac{[s]}{qs},~\frac{1}{b_{k}} :=\frac{[s]}{rs}+\left(k-\frac{[s]}{s}\right)\frac{1}{q},
\end{equation}
and $p_{1_k}$ is given in \eqref{GrindEQ__3_18_}.

First, we estimate $I_{k_1}$. If $k=1$, we can see that $\left|\alpha _{1} \right|=[s]$ and $a_{k} =r$. Hence, we immediately get
\[\left\| D^{\alpha _{1}} u\right\| _{\dot{H}_{a_{k}}^{v} } \lesssim\left\| u\right\| _{\dot{H}_{r}^{s} } .\]
We consider the case $k>1$. For $1\le i\le k$, putting
\begin{equation} \label{GrindEQ__3_31_}
\frac{1}{\tilde{\alpha}_{k_i}}:=\frac{|\alpha_{i}|+v}{s}\frac{1}{r}+\left(1-\frac{|\alpha_{i}|+v}{s}\right)\frac{1}{q},
\end{equation}
we can see that
\begin{equation} \label{GrindEQ__3_32_}
\frac{1}{a_{k} } =\sum _{j\in I_{k}^{i} }\frac{1}{\alpha_{k_j} }  +\frac{1}{\tilde{\alpha}_{k_i} },
\end{equation}
where $I_{k}^{i} =\left\{j\in \mathbb N:\;1\le j\le k,\;j\ne i\right\}$ and $\alpha_{k_j}$ is given in \eqref{GrindEQ__3_18_}.
By using Corollary \ref{cor 2.6.}, Lemma \ref{lem 2.8.} and \eqref{GrindEQ__3_32_}, we have
\begin{eqnarray}\begin{split} \label{GrindEQ__3_33_}
\left\| \prod _{i=1}^{k}D^{\alpha _{i} } u \right\| _{\dot{H}_{a_k}^{v}}
&\lesssim \sum _{i=1}^{k}(\left\| D^{\alpha _{i} } u\right\| _{\dot{H}_{\tilde{\alpha}_{k_i}}^{v} } \prod _{j\in I_{k}^{i} }\left\| D^{\alpha _{j} } u \right\| _{L^{\alpha_{k_j}}})\\
&\lesssim\sum _{i=1}^{k}(\left\| u\right\| _{\dot{H}_{\tilde{\alpha}_{k_i}}^{|\alpha _{i}|+v} } \prod _{j\in I_{k}^{i} }\left\| u\right\| _{\dot{H}_{\alpha_{k_j}}^{|\alpha _{j}|} })\\
&\lesssim \sum _{i=1}^{k} \left[\left\| u\right\| _{\dot{H}_{r}^{s} }^{\frac{|\alpha _{i}|+v}{s}}\left\|u\right\|_{L^{q}}^{1-\frac{|\alpha _{i}|+v}{s}}
\prod _{j\in I_{k}^{i} }(\left\| u\right\| _{\dot{H}_{r}^{s} }^{\frac{|\alpha _{j}|}{s}}\left\|u\right\|_{L^{q}}^{1-\frac{|\alpha _{j}|}{s}})\right]\\
&=\left\| u\right\| _{\dot{H}_{r}^{s} }\left\|u\right\|_{L^{q}}^{k-1}.
\end{split}\end{eqnarray}
Hence, for any $1\le k\le [s]$, we have
\begin{equation} \label{GrindEQ__3_34_}
I_{k_1} \lesssim \left\| f^{(k)} (u)-f^{(k)} (v)\right\| _{L^{p_{1_k}}} \left\| u\right\| _{\dot{H}_{r}^{s} }\left\|u\right\|_{L^{q}}^{k-1}.
\end{equation}
Since $k\le [s]<\left\lceil s\right\rceil $, it follows from \eqref{GrindEQ__3_10_} and H\"{o}lder inequality that
\begin{eqnarray}\begin{split} \label{GrindEQ__3_35_}
\left\| f^{(k)} (u)-f^{(k)} (v)\right\| _{L^{p_{1_k}}}&=\left\| \left(u-v\right)\int _{0}^{1}f^{\left(k+1\right)} \left(v+t\left(u-v\right)\right)dt \right\| _{L^{p_{1_k}}}\\
&\lesssim \left\| \left(u-v\right)\int _{0}^{1}\left|v+t\left(u-v\right)\right|^{\sigma -k} dt \right\| _{L^{p_{1_k}}}\\
&\lesssim \left\| u-v\right\| _{L^{q} } (\left\| u\right\| _{L^{q}}^{\sigma -k} +\left\| v\right\| _{L^{q}}^{\sigma -k}).
\end{split}\end{eqnarray}
In view of \eqref{GrindEQ__3_34_} and \eqref{GrindEQ__3_35_}, we immediately get
\begin{eqnarray}\begin{split} \label{GrindEQ__3_36_}
I_{k_1}\lesssim
(\left\| u\right\| _{L^q}^{\sigma-1} +\left\| v\right\| _{L^q}^{\sigma-1})
\left\| u\right\| _{\dot{H}_{r}^{s}}\left\| u-v\right\| _{L^{q}}.
\end{split}\end{eqnarray}

Next, we estimate $I_{k_2}$. Using \eqref{GrindEQ__3_18_}, \eqref{GrindEQ__3_30_}, H\"{o}lder inequality and Lemma \ref{lem 2.8.}, we have
\begin{eqnarray}\begin{split} \label{GrindEQ__3_37_}
\left\| \prod _{i=1}^{k}D^{\alpha _{i} } u \right\| _{L^{b_k}}
&\lesssim \prod _{i=1}^{k} \left\| D^{\alpha _{i} } u \right\| _{L^{\alpha_{k_i}}}
\lesssim \prod _{i=1}^{k}(\left\| u\right\| _{\dot{H}_{r}^{s} }^{\frac{|\alpha _{i}|}{s}}\left\|u\right\|_{L^{q}}^{1-\frac{|\alpha _{i}|}{s}})
=\left\| u\right\| _{\dot{H}_{r}^{s} }^{\frac{[s]}{s}}\left\|u\right\|_{L^{q}}^{k-\frac{[s]}{s}}.
\end{split}\end{eqnarray}
It also follows from \eqref{GrindEQ__3_30_} and Lemma \ref{lem 2.8.} that
\begin{eqnarray}\begin{split} \label{GrindEQ__3_38_}
\left\| f^{(k)} (u)-f^{(k)} (v)\right\| _{\dot{H}_{p_{2_k} }^{v} }&\le \left\| f^{(k)} (u)-f^{(k)} (v)\right\| _{L^{p_{1_k}}}^{1-v}
\left\| f^{(k)} (u)-f^{(k)} (v)\right\| _{\dot{H}_{p_{3_k} }^{1} }^{v},
\end{split}\end{eqnarray}
where
\begin{eqnarray}\begin{split} \label{GrindEQ__3_39_}
\frac{1}{p_{3_k}}:=\left(\sigma-k+1-\frac{1}{s}\right)\frac{1}{q}+\frac{1}{rs}.
\end{split}\end{eqnarray}
Using Lemma \ref{lem 2.7.}, we can get
\begin{eqnarray}\begin{split} \label{GrindEQ__3_40_}
\left\| f^{(k)} (u)-f^{(k)} (v)\right\| _{\dot{H}_{p_{3_k} }^{1} }&\le \sum_{j=1}^{d}\left\| f^{(k+1)}(u)\partial_{j}u-f^{(k+1)} (v)\partial_{j}v\right\| _{L^{p_{3_k} }}.
\end{split}\end{eqnarray}
Putting
\begin{eqnarray}\begin{split} \label{GrindEQ__3_41_}
\frac{1}{p_{4_k}}:=\frac{\sigma-k}{q},~\frac{1}{p_{5_k}}:=\frac{1}{rs}+\left(1-\frac{1}{s}\right)\frac{1}{q},
\end{split}\end{eqnarray}
it follows from \eqref{GrindEQ__3_10_}, \eqref{GrindEQ__3_39_}, H\"{o}lder inequality and Lemma \ref{lem 2.8.} that
\begin{eqnarray}\begin{split} \label{GrindEQ__3_42_}
&\left\| f^{(k+1)}(u)\partial_{j}u-f^{(k+1)} (v)\partial_{j}v\right\| _{L^{p_{3_k} }}\\
&~~~\le \left\| f^{(k+1)}(u)-f^{(k+1)}(v)\right\|_{L^{p_{4_{k}}}}\left\|\partial_{j}u\right\| _{L^{p_{5_k} }}+
\left\| f^{(k+1)}(v)\right\|_{L^{p_{4_{k}}}}\left\|\partial_{j}u-\partial_{j}v\right\| _{L^{p_{5_k}}}\\
&~~~\le \left\| f^{(k+1)}(u)-f^{(k+1)}(v)\right\|_{L^{p_{4_{k}}}}\left\|u\right\| _{\dot{H}_{r}^{s}}^{\frac{1}{s}}\left\|u\right\|_{L^q}^{1-\frac{1}{s}}+
\left\|v\right\|_{L^{q}}^{\sigma-k}\left\|u-v\right\| _{\dot{H}_{r}^{s}}^{\frac{1}{s}}\left\|u-v\right\| _{L^{q}}^{1-\frac{1}{s}}.
\end{split}\end{eqnarray}

$\cdot$ If $k<[s]$, then it follows from \eqref{GrindEQ__3_10_} and H\"{o}lder inequality
\begin{eqnarray}\begin{split} \label{GrindEQ__3_43_}
\left\| f^{(k+1)} (u)-f^{(k+1)} (v)\right\| _{L^{p_{4_k}}}&=\left\| \left(u-v\right)\int _{0}^{1}f^{(k+2)} \left(v+t\left(u-v\right)\right)dt \right\| _{L^{p_{4_k}}}\\
&\lesssim \left\| \left(u-v\right)\int _{0}^{1}\left|v+t\left(u-v\right)\right|^{\sigma -k-1} dt \right\| _{L^{p_{4_k}}}\\
&\lesssim \left\| u-v\right\| _{L^{q} } (\left\| u\right\| _{L^{q}}^{\sigma -k-1} +\left\| v\right\| _{L^{q}}^{\sigma -k-1}).
\end{split}\end{eqnarray}
In view of \eqref{GrindEQ__3_35_}, \eqref{GrindEQ__3_37_}--\eqref{GrindEQ__3_43_}, we have
\begin{eqnarray}\begin{split}\label{GrindEQ__3_44_}
I_{k_2}&=\left\| f^{(k)} (u)-f^{(k)} (v)\right\| _{\dot{H}_{p_{2_k} }^{v} } \left\| \prod _{i=1}^{k}D^{\alpha _{i} } u \right\| _{b_{k} }\\
&\lesssim \left\| u\right\| _{\dot{H}_{r}^{s} }^{\frac{[s]}{s}}\left\|u\right\|_{L^{q}}^{k-\frac{[s]}{s}}
\left\| f^{(k)} (u)-f^{(k)} (v)\right\| _{L^{p_{1_k}}}^{1-v}
(A+B)^{v}\\
&\lesssim \left\| u\right\| _{\dot{H}_{r}^{s} }^{\frac{[s]}{s}}\left\|u\right\|_{L^{q}}^{k-\frac{[s]}{s}}
\left[\left\| u-v\right\| _{L^{q} } (\left\| u\right\| _{L^{q}}^{\sigma -k} +\left\| v\right\| _{L^{q}}^{\sigma -k})\right]^{1-v}
(A^v+B^v),
\end{split}\end{eqnarray}
where
\begin{eqnarray}\begin{split}\label{GrindEQ__3_45_}
A:=\left\| u-v\right\| _{L^{q} } (\left\| u\right\| _{L^{q}}^{\sigma -k-\frac{1}{s}}+\left\| v\right\| _{L^{q}}^{\sigma -k-\frac{1}{s}})\left\|u\right\| _{\dot{H}_{r}^{s}}^{\frac{1}{s}},~B:=\left\|v\right\|_{L^{q}}^{\sigma-k}\left\|u-v\right\| _{\dot{H}_{r}^{s}}^{\frac{1}{s}}\left\|u-v\right\| _{L^{q}}^{1-\frac{1}{s}}.
\end{split}\end{eqnarray}
An easy computation shows that
\begin{eqnarray}\begin{split}\label{GrindEQ__3_46_}
\left\| u\right\| _{\dot{H}_{r}^{s} }^{\frac{[s]}{s}}\left\|u\right\|_{L^{q}}^{k-\frac{[s]}{s}}
&\left[\left\| u-v\right\| _{L^{q} } (\left\| u\right\| _{L^{q}}^{\sigma -k} +\left\| v\right\| _{L^{q}}^{\sigma -k})\right]^{1-v}A^v\\
&\lesssim \left\| u\right\| _{\dot{H}_{r}^{s} }(\left\| u\right\| _{L^{q}}^{\sigma -1} +\left\| v\right\| _{L^{q}}^{\sigma -1})\left\| u-v\right\| _{L^{q} }
\end{split}\end{eqnarray}
and
\begin{eqnarray}\begin{split}\label{GrindEQ__3_47_}
\left\| u\right\| _{\dot{H}_{r}^{s} }^{\frac{[s]}{s}}&\left\|u\right\|_{L^{q}}^{k-\frac{[s]}{s}}
\left[\left\| u-v\right\| _{L^{q} } (\left\| u\right\| _{L^{q}}^{\sigma -k} +\left\| v\right\| _{L^{q}}^{\sigma -k})\right]^{1-v}B^v\\
&\lesssim (\left\| u\right\| _{L^{q}}^{\sigma -1} +\left\| v\right\| _{L^{q}}^{\sigma -1})(\left\|u\right\|_{L^{q}}\left\| u-v\right\| _{\dot{H}_{r}^{s} })^{\frac{v}{s}}(\left\| u\right\| _{\dot{H}_{r}^{s} }\left\| u-v\right\| _{L^{q}})^{\frac{[s]}{s}}\\
&\lesssim  (\left\| u\right\| _{L^{q}}^{\sigma -1} +\left\| v\right\| _{L^{q}}^{\sigma -1})(\left\|u\right\|_{L^{q}}\left\| u-v\right\| _{\dot{H}_{r}^{s} }+\left\| u\right\| _{\dot{H}_{r}^{s} }\left\| u-v\right\| _{L^{q}})\\
&\lesssim (\left\| u\right\| _{L^{q}}^{\sigma} +\left\| v\right\| _{L^{q}}^{\sigma})\left\| u-v\right\| _{\dot{H}_{r}^{s} }
+(\left\| u\right\| _{L^{q}}^{\sigma -1} +\left\| v\right\| _{L^{q}}^{\sigma -1})\left\| u\right\| _{\dot{H}_{r}^{s} }\left\| u-v\right\| _{L^{q}}.
\end{split}\end{eqnarray}
\eqref{GrindEQ__3_44_}--\eqref{GrindEQ__3_47_} yield
\begin{eqnarray}\begin{split}\label{GrindEQ__3_48_}
I_{k_2}\lesssim (\left\| u\right\| _{L^{q}}^{\sigma} +\left\| v\right\| _{L^{q}}^{\sigma})\left\| u-v\right\| _{\dot{H}_{r}^{s} }
+(\left\| u\right\| _{L^{q}}^{\sigma -1} +\left\| v\right\| _{L^{q}}^{\sigma -1})\left\| u\right\| _{\dot{H}_{r}^{s} }\left\| u-v\right\| _{L^{q}}.
\end{split}\end{eqnarray}
In view of \eqref{GrindEQ__3_29_}, \eqref{GrindEQ__3_36_} and \eqref{GrindEQ__3_48_}, we immediately get
 \begin{eqnarray}\begin{split}\label{GrindEQ__3_49_}
\left\| I_{k} \right\| _{\dot{H}_{p}^{v} }\lesssim (\left\| u\right\| _{L^{q}}^{\sigma} +\left\| v\right\| _{L^{q}}^{\sigma})\left\| u-v\right\| _{\dot{H}_{r}^{s} }
+(\left\| u\right\| _{L^{q}}^{\sigma -1} +\left\| v\right\| _{L^{q}}^{\sigma -1})\left\| u\right\| _{\dot{H}_{r}^{s} }\left\| u-v\right\| _{L^{q}},
\end{split}\end{eqnarray}
for any $k<[s]$.

$\cdot$ If $k=[s]$, then it follows from \eqref{GrindEQ__3_11_} and H\"{o}lder inequality that
\begin{eqnarray}\begin{split} \label{GrindEQ__3_50_}
\left\| f^{(k+1)} (u)-f^{(k+1)} (v)\right\| _{L^{p_{4_k}}}&\lesssim\left\| \left|u-v\right|^{\min \{ \sigma -\left\lceil s\right\rceil +1,\;1\} } \left(|u|+|v|\right)^{\max \{ 0,\;\sigma -\left\lceil s\right\rceil \} } \right\| _{L^{p_{4_k}}}  \\
&\lesssim\left\| u-v\right\| _{L^q}^{\min \{ \sigma -\left\lceil s\right\rceil +1,\;1\} } (\left\| u\right\|_{L^q}+\left\|v\right\|_{L^q})^{\max \{ 0,\;\sigma -\left\lceil s\right\rceil \} }.
\end{split}\end{eqnarray}
If $\sigma\ge \left\lceil s\right\rceil$, we can repeat the same argument as in the case $k<[s]$ to obtain \eqref{GrindEQ__3_49_}. Hence, it suffices to consider the case $\left \lceil s\right\rceil-1\le \sigma<\left\lceil s\right\rceil$.
Notice that
\begin{eqnarray}\begin{split} \label{GrindEQ__3_51_}
(|z_1|^{c-1}+|z_2|^{c-1})|z_1-z_2|\lesssim |z_1-z_2|^c,~\textrm{for}~0<c<1~\textrm{and}~z_1,~z_2\in \mathbb C.
\end{split}\end{eqnarray}
In fact, with out loss of generality, we assume $|z_1|\ge|z_2|>0$. If $|z_1-z_2|\le |z_1|$, then (since $c<1$)
$$(|z_1|^{c-1}+|z_2|^{c-1})|z_1-z_2|\le 2|z_1|^{c-1}|z_1-z_2|\le 2|z_1-z_2|^{c}.$$
If $|z_1-z_2|\ge |z_1|$, then (since $|z_1-z_2|\le |z_1|+|z_2|\le 2|z_1|$) we see that
$$(|z_1|^{c-1}+|z_2|^{c-1})|z_1-z_2|\le 2|z_1|^{c-1}|z_1-z_2|\le 4|z_1|^{c}\le 4|z_1-z_2|^{c}.$$
It follows from \eqref{GrindEQ__3_10_}, \eqref{GrindEQ__3_18_}, \eqref{GrindEQ__3_51_} and H\"{o}lder inequality that
\begin{eqnarray}\begin{split} \label{GrindEQ__3_52_}
\left\| f^{([s])} (u)-f^{([s])} (v)\right\| _{L^{p_{1_k}}}&=\left\| \left(u-v\right)\int _{0}^{1}f^{(\lceil s\rceil)} \left(v+t\left(u-v\right)\right)dt \right\| _{L^{p_{1_k}}}\\
&\lesssim \left\| \left(u-v\right)\int _{0}^{1}\left|v+t\left(u-v\right)\right|^{\sigma -[s]} dt \right\| _{L^{p_{1_k}}}\\
&\lesssim \left\| |u-v|(|u|^{\sigma -[s]}+|v|^{\sigma -[s]})\right\| _{L^{p_{1_k}}}\\
&\lesssim \left\| |u-v|(|u|^{\sigma -[s]-1}+|v|^{\sigma -[s]-1})(|u|+|v|)\right\| _{L^{p_{1_k}}}\\
&\lesssim \left\| |u-v|^{\sigma -[s]}(|u|+|v|)\right\| _{L^{p_{1_k}}}\\
&\lesssim \left\|u-v\right\|_{L^q}^{\sigma -[s]}(\left\|u\right\|_{L^q}+\left\|u\right\|_{L^q})
\end{split}\end{eqnarray}
In view of \eqref{GrindEQ__3_35_}, \eqref{GrindEQ__3_37_}--\eqref{GrindEQ__3_42_}, \eqref{GrindEQ__3_50_} and \eqref{GrindEQ__3_52_}, we have
\begin{eqnarray}\begin{split}\label{GrindEQ__3_53_}
I_{k_2}&=\left\| f^{(k)} (u)-f^{(k)} (v)\right\| _{\dot{H}_{p_{2_k} }^{v} } \left\| \prod _{i=1}^{k}D^{\alpha _{i} } u \right\| _{b_{k} }\\
&\lesssim \left\| u\right\| _{\dot{H}_{r}^{s} }^{\frac{[s]}{s}}\left\|u\right\|_{L^{q}}^{[s]-\frac{[s]}{s}}
\left\| f^{(k)} (u)-f^{(k)} (v)\right\| _{L^{p_{1_k}}}^{1-v}
(C+B)^{v}\\
&\lesssim \left\| u\right\| _{\dot{H}_{r}^{s} }^{\frac{[s]}{s}}\left\|u\right\|_{L^{q}}^{[s]-\frac{[s]}{s}}
\left[\left\|u-v\right\|_{L^q}^{\sigma -[s]-1}(\left\|u\right\|_{L^q}+\left\|u\right\|_{L^q})\right]^{1-v}
C^v\\
&~~+\left\| u\right\| _{\dot{H}_{r}^{s} }^{\frac{[s]}{s}}\left\|u\right\|_{L^{q}}^{[s]-\frac{[s]}{s}}
\left[\left\| u-v\right\| _{L^{q} } (\left\| u\right\| _{L^{q}}^{\sigma -[s]} +\left\| v\right\| _{L^{q}}^{\sigma -[s]})\right]^{1-v}
B^v,
\end{split}\end{eqnarray}
where $B$ is given in \eqref{GrindEQ__3_45_} and
\begin{eqnarray}\begin{split}\label{GrindEQ__3_54_}
C:=\left\| u-v\right\| _{L^q}^{\sigma -[s]}\left\|u\right\| _{\dot{H}_{r}^{s}}^{\frac{1}{s}}\left\|u\right\| _{L^{q}}^{1-\frac{1}{s}}.
\end{split}\end{eqnarray}
An easy computation shows that
\begin{eqnarray}\begin{split}\label{GrindEQ__3_55_}
\left\| u\right\| _{\dot{H}_{r}^{s} }^{\frac{[s]}{s}}\left\|u\right\|_{L^{q}}^{[s]-\frac{[s]}{s}}
\left[\left\|u-v\right\|_{L^q}^{\sigma -[s]-1}(\left\|u\right\|_{L^q}+\left\|u\right\|_{L^q})\right]^{1-v}
C^v\\
\lesssim \left\|u\right\|_{\dot{H}_{r}^{s}}(\left\|u\right\|_{L^q}^{[s]}+\left\|u\right\|_{L^q}^{[s]})\left\|u-v\right\|_{L^q}^{\sigma-[s]}.
\end{split}\end{eqnarray}
In view of \eqref{GrindEQ__3_47_}, \eqref{GrindEQ__3_53_} and \eqref{GrindEQ__3_55_}, we have
\begin{eqnarray}\begin{split}\label{GrindEQ__3_56_}
I_{k_2}&\lesssim (\left\| u\right\| _{L^{q}}^{\sigma} +\left\| v\right\| _{L^{q}}^{\sigma})\left\| u-v\right\| _{\dot{H}_{r}^{s} }
+(\left\| u\right\| _{L^{q}}^{\sigma -1} +\left\| v\right\| _{L^{q}}^{\sigma -1})\left\| u\right\| _{\dot{H}_{r}^{s} }\left\| u-v\right\| _{L^{q}}\\
&~~+\left\|u\right\|_{\dot{H}_{r}^{s}}(\left\|u\right\|_{L^q}^{[s]}+\left\|u\right\|_{L^q}^{[s]})\left\|u-v\right\|_{L^q}^{\sigma-[s]},
\end{split}\end{eqnarray}
for $k=[s]$ and $\lceil s\rceil-1 \le \sigma<\lceil s\rceil$.
Collecting the above three cases ($k<[s]$; $k=[s]$ and $\sigma\ge \lceil s\rceil$; $k=[s]$ and $\lceil s\rceil-1 \le \sigma<\lceil s\rceil$),
we have
\begin{eqnarray}\begin{split}\label{GrindEQ__3_57_}
I_{k_2}&\lesssim(\left\| u\right\| _{L^q}^{\sigma } +\left\| v\right\| _{L^q}^{\sigma })\left\| u-v\right\| _{\dot{H}_{r}^{s} }
+(\left\| u\right\| _{L^q}^{\sigma-1} +\left\| v\right\| _{L^q}^{\sigma-1})
\left\| u\right\| _{\dot{H}_{r}^{s}}\left\| u-v\right\| _{L^{q}}\\
&~~+(\left\| u\right\| _{L^q}^{\max \{ \left\lceil s\right\rceil-1 ,\;\sigma-1\} } +\left\| v\right\| _{L^q}^{\max \{ \left\lceil s\right\rceil-1,\;\sigma-1 \} } )\left\|u\right\| _{\dot{H}_{r}^{s}}\left\| u-v\right\| _{L^q}^{\min \{ \sigma -\left\lceil s\right\rceil +1,\;1\} },
\end{split}\end{eqnarray}
for any $1\le k\le [s]$.
\eqref{GrindEQ__3_36_} and \eqref{GrindEQ__3_57_} imply that
\begin{eqnarray}\begin{split}\label{GrindEQ__3_58_}
I_{k}&\lesssim(\left\| u\right\| _{L^q}^{\sigma } +\left\| v\right\| _{L^q}^{\sigma })\left\| u-v\right\| _{\dot{H}_{r}^{s} }
+(\left\| u\right\| _{L^q}^{\sigma-1} +\left\| v\right\| _{L^q}^{\sigma-1})
\left\| u\right\| _{\dot{H}_{r}^{s}}\left\| u-v\right\| _{L^{q}}\\
&~~+(\left\| u\right\| _{L^q}^{\max \{ \left\lceil s\right\rceil-1 ,\;\sigma-1\} } +\left\| v\right\| _{L^q}^{\max \{ \left\lceil s\right\rceil-1,\;\sigma-1 \} } )\left\|u\right\| _{\dot{H}_{r}^{s}}\left\| u-v\right\| _{L^q}^{\min \{ \sigma -\left\lceil s\right\rceil +1,\;1\} },
\end{split}\end{eqnarray}
for any $1\le k\le [s]$.

\emph{Step 2.2.} Next, we estimate $\left\| II_{k} \right\| _{\dot{H}_{p}^{v} } $, where $II_{k}$ is given in \eqref{GrindEQ__3_16_}.
Lemma \ref{lem 2.5.} (fractional product rule) yields that
\begin{eqnarray}\begin{split} \label{GrindEQ__3_59_}
\left\| II_{k} \right\| _{\dot{H}_{p}^{v} } &\lesssim\left\| f^{(k)} (v)\right\| _{p_{1_k} } \left\| \prod _{i=1}^{k}D^{\alpha _{i} } u -\prod _{i=1}^{k}D^{\alpha _{i} } v\right\| _{\dot{H}_{a_{k} }^{v} } \\
&~~~~~~~~+\left\| f^{(k)} (v)\right\| _{\dot{H}_{p_{2_k} }^{v} } \left\| \prod _{i=1}^{k}D^{\alpha _{i} } u-\prod _{i=1}^{k}D^{\alpha _{i} } v \right\| _{b_{k} }\\
&\equiv II_{k_1} +II_{k_2} ,
\end{split}\end{eqnarray}
where $p_{1_k}$ is given in \eqref{GrindEQ__3_18_} and $a_{k}$, $p_{2_k} $, $b_{k}$ are given in \eqref{GrindEQ__3_30_}.

First, we estimate $II_{k_1}$. If $k=1$, it follows from \eqref{GrindEQ__3_10_} that
\begin{eqnarray}\begin{split} \label{GrindEQ__3_60_}
II_{k_1}\lesssim \left\| f'(v)\right\| _{p_{1_k} }\left\| D^{\alpha _{1}} u-D^{\alpha _{1}} v\right\| _{\dot{H}_{a_{k}}^{v} } \lesssim \left\|v\right\|_{L^q}^{\sigma}\left\| u-v\right\| _{\dot{H}_{r}^{s} } .
\end{split}\end{eqnarray}
We consider the case $k>1$. In view of \eqref{GrindEQ__3_25_}, we have
\begin{eqnarray}\begin{split} \label{GrindEQ__3_61_}
&II_{k_1}\lesssim \left\| f^{(k)} (v)\right\| _{L^{p_{1_k}}}\sum_{i=1}^{k}{\left\| \prod _{j=1}^{i-1}D^{\alpha _{j} } u \prod _{j=i+1}^{k}D^{\alpha _{j} } v \left(D^{\alpha _{i} } u-D^{\alpha _{i} } v\right)\right\| _{\dot{H}_{a_k}^{v}}}.
\end{split}\end{eqnarray}
Combining the argument in the proof of \eqref{GrindEQ__3_27_} with that in the proof of \eqref{GrindEQ__3_33_}, we have
\begin{eqnarray}\begin{split} \label{GrindEQ__3_62_}
&\left\| f^{(k)} (v)\right\| _{L^{p_{1_k}}}\left\| \prod _{j=1}^{i-1}D^{\alpha _{j} } u \prod _{j=i+1}^{k}D^{\alpha _{j} } v \left(D^{\alpha _{i} } u-D^{\alpha _{i} } v\right)\right\| _{\dot{H}_{a_k}^{v}}\\
&\lesssim \left\|v\right\|_{L^q}^{\sigma+1-k}\left\|D^{\alpha _{i} } u-D^{\alpha _{i} } v\right\| _{\dot{H}_{\tilde{\alpha}_{k_i}}^{v}}
\left\|\prod _{j=1}^{i-1}D^{\alpha _{j} } u \prod _{j=i+1}^{k}D^{\alpha _{j} } v\right\| _{L^{\tilde{b}_{k_i}}}\\
&~~+\left\|v\right\|_{L^q}^{\sigma+1-k}\left\|D^{\alpha _{i} } u-D^{\alpha _{i} } v\right\| _{L^{\alpha_{k_{i}}}}
\left\|\prod _{j=1}^{i-1}D^{\alpha _{j} } u \prod _{j=i+1}^{k}D^{\alpha _{j} } v\right\| _{\dot{H}_{b_{k_i}}^v}\\
&\lesssim\left[(\left\|u\right\| _{\dot{H}_{r}^{s}}+\left\|v\right\| _{\dot{H}_{r}^{s}})\left\|u-v\right\|_{L^q}\right]^{1-\frac{|\tilde{\alpha}_{i}|}{s}}
\left[\left\|u-v\right\| _{\dot{H}_{r}^{s}}(\left\|u\right\|_{L^q}+\left\|v\right\|_{L^q})\right]^{\frac{|\tilde{\alpha}_{i}|}{s}}
(\left\|u\right\|_{L^q}+\left\|v\right\|_{L^q})^{\sigma-1}\\
&~~+ \left[(\left\|u\right\| _{\dot{H}_{r}^{s}}+\left\|v\right\| _{\dot{H}_{r}^{s}})\left\|u-v\right\|_{L^q}\right]^{1-\frac{|\alpha_{i}|}{s}}
\left[\left\|u-v\right\| _{\dot{H}_{r}^{s}}(\left\|u\right\|_{L^q}+\left\|v\right\|_{L^q})\right]^{\frac{|\alpha_{i}|}{s}}
(\left\|u\right\|_{L^q}+\left\|v\right\|_{L^q})^{\sigma-1}\\
&\lesssim(\left\| u\right\| _{L^q}^{\sigma } +\left\| v\right\| _{L^q}^{\sigma })\left\| u-v\right\| _{\dot{H}_{r}^{s}}+
(\left\| u\right\| _{L^q}^{\sigma-1} +\left\| v\right\| _{L^q}^{\sigma-1})
(\left\| u\right\| _{\dot{H}_{r}^{s}}+\left\| v\right\| _{\dot{H}_{r}^{s}})\left\| u-v\right\| _{L^{q}},
\end{split}\end{eqnarray}
where
$$
\frac{1}{\tilde{b}_{k_i}}:=\frac{1}{a_k}-\frac{1}{\tilde{\alpha}_{k_i}},~\frac{1}{b_{k_i}}:=\frac{1}{a_k}-\frac{1}{\alpha_{k_i}},
$$
and $\alpha_{k_i}$, $\tilde{\alpha}_{k_i}$ are given in \eqref{GrindEQ__3_18_} and \eqref{GrindEQ__3_31_} respectively.
In view of \eqref{GrindEQ__3_60_}--\eqref{GrindEQ__3_62_}, we have
\begin{eqnarray}\begin{split} \label{GrindEQ__3_63_}
I_{k_1}\lesssim(\left\| u\right\| _{L^q}^{\sigma } +\left\| v\right\| _{L^q}^{\sigma })\left\| u-v\right\| _{\dot{H}_{r}^{s}}+
(\left\| u\right\| _{L^q}^{\sigma-1} +\left\| v\right\| _{L^q}^{\sigma-1})
(\left\| u\right\| _{\dot{H}_{r}^{s}}+\left\| v\right\| _{\dot{H}_{r}^{s}})\left\| u-v\right\| _{L^{q}},
\end{split}\end{eqnarray}
for any $1\le k\le [s]$.

Next, we estimate $II_{k_2}$. If follows from \eqref{GrindEQ__3_30_}, Lemma \ref{lem 2.4.} (fractional chain rule) and Lemma \ref{lem 2.8.} that
\begin{eqnarray}\begin{split} \label{GrindEQ__3_64_}
\left\| f^{(k)}(v)\right\| _{\dot{H}_{p_{2_k} }^{v} }&\lesssim \left\|f^{(k+1)}(v)\right\|_{L^{p_{6_k}}}\left\|v\right\|_{\dot{H}_r^s}
\lesssim \left\|v\right\|_{L^q}^{\sigma-k}\left\|v\right\|_{\dot{H}_r^s}^{\frac{v}{s}}\left\|v\right\|_{L^q}^{\frac{[s]}{s}},
\end{split}\end{eqnarray}
where $\frac{1}{p_{6_k}}:=\frac{\sigma-k}{q}$ and $\frac{1}{p_{7_k}}:=\frac{v}{rs}+\frac{[s]}{qs}$.
If $k=1$, if follows from \eqref{GrindEQ__3_30_}, \eqref{GrindEQ__3_64_} and Lemma \ref{lem 2.8.} that
\begin{eqnarray}\begin{split} \label{GrindEQ__3_65_}
II_{k_2}&= \left\| f'(v)\right\| _{\dot{H}_{p_{2_k} }^{v} } \left\|D^{\alpha _{1} } u-D^{\alpha _{1} } v \right\| _{b_{k}}
\lesssim \left\|v\right\|_{L^q}^{\sigma-k}\left\|v\right\|_{\dot{H}_r^s}^{\frac{v}{s}}\left\|v\right\|_{L^q}^{\frac{[s]}{s}}
\left\|u-v\right\|_{\dot{H}_r^s}^{\frac{[s]}{s}}\left\|u-v\right\|_{L^q}^{\frac{v}{s}}\\
&\lesssim \left\|v\right\|_{L^q}^{\sigma-1}(\left\|v\right\|_{\dot{H}_r^s}\left\|u-v\right\|_{L^q}^{\frac{v}{s}}+\left\|v\right\|_{L^q}
\left\|u-v\right\|_{\dot{H}_r^s}).
\end{split}\end{eqnarray}
Next, we consider the case $k\ge 2$. In view of \eqref{GrindEQ__3_25_}, we have
\begin{eqnarray}\begin{split} \label{GrindEQ__3_66_}
&II_{k_2}\lesssim \left\| f^{(k)}(v)\right\| _{\dot{H}_{p_{2_k} }^{v} }\sum_{i=1}^{k}{\left\| \prod _{j=1}^{i-1}D^{\alpha _{j} } u \prod _{j=i+1}^{k}D^{\alpha _{j} } v \left(D^{\alpha _{i} } u-D^{\alpha _{i} } v\right)\right\| _{L^{b_k}}}.
\end{split}\end{eqnarray}
Using \eqref{GrindEQ__3_18_}, \eqref{GrindEQ__3_30_}, \eqref{GrindEQ__3_64_}, and Lemma \ref{lem 2.8.}, we have
\begin{eqnarray}\begin{split} \label{GrindEQ__3_67_}
\left\| f^{(k)}(v)\right\| _{\dot{H}_{p_{2_k} }^{v} }&\left\| \prod _{j=1}^{i-1}D^{\alpha _{j} } u \prod _{j=i+1}^{k}D^{\alpha _{j} } v \left(D^{\alpha _{i} } u-D^{\alpha _{i} } v\right)\right\| _{L^{b_k}}\\
&\lesssim \left\|v\right\|_{L^q}^{\sigma-k}\left\|v\right\|_{\dot{H}_r^s}^{\frac{v}{s}}\left\|v\right\|_{L^q}^{\frac{[s]}{s}} \left\| u-v\right\| _{\dot{H}_{\alpha_{k_i} }^{|\alpha _{i}|} } \prod _{j=1}^{i-1}\left\| u\right\| _{\dot{H}_{\alpha_{k_j} }^{|\alpha _{j}|} }  \prod _{j=i+1}^{k}\left\| v\right\| _{\dot{H}_{\alpha_{k_j} }^{|\alpha _{j}|} }\\
&\lesssim \left\|v\right\|_{L^q}^{\sigma-k}\left\|v\right\|_{\dot{H}_r^s}^{\frac{v}{s}}\left\|v\right\|_{L^q}^{\frac{[s]}{s}} \left\|u-v\right\| _{\dot{H}_{r}^{s}}^{\frac{|\alpha_{i}|}{s}}\left\|u-v\right\| _{L^{q}}^{1-\frac{|\alpha_{i}|}{s}}\\
&~~~~~~~~~\cdot\prod_{j=1}^{i-1} (\left\|u\right\| _{\dot{H}_{r}^{s}}^{\frac{|\alpha_{j}|}{s}}\left\|u\right\| _{L^{q}}^{1-\frac{|\alpha_{j}|}{s}})
\prod_{j=i+1}^{k} (\left\|v\right\| _{\dot{H}_{r}^{s}}^{\frac{|\alpha_{j}|}{s}}\left\|v\right\| _{L^{q}}^{1-\frac{|\alpha_{j}|}{s}})\\
&\lesssim (\left\|u\right\|_{L^q}+\left\|v\right\|_{L^q})^{\sigma-1}
\left[(\left\|u\right\| _{\dot{H}_{r}^{s}}+\left\|v\right\| _{\dot{H}_{r}^{s}})\left\|u-v\right\|_{L^q}\right]^{1-\frac{|\alpha_{i}|}{s}}\\
&~~~~~~~~~\cdot\left[\left\|u-v\right\| _{\dot{H}_{r}^{s}}(\left\|u\right\|_{L^q}+\left\|v\right\|_{L^q})\right]^{\frac{|\alpha_{i}|}{s}}\\
&\lesssim (\left\| u\right\| _{L^q}^{\sigma-1} +\left\| v\right\| _{L^q}^{\sigma-1})
(\left\| u\right\| _{\dot{H}_{r}^{s}}+\left\| v\right\| _{\dot{H}_{r}^{s}})\left\| u-v\right\| _{L^{q}}\\
&~~~~~~~~~+(\left\| u\right\| _{L^q}^{\sigma } +\left\| v\right\| _{L^q}^{\sigma })\left\| u-v\right\| _{\dot{H}_{r}^{s}},
\end{split}\end{eqnarray}
In view of \eqref{GrindEQ__3_65_}--\eqref{GrindEQ__3_67_}, we have
\begin{eqnarray}\begin{split} \label{GrindEQ__3_68_}
II_{k_2}\lesssim(\left\| u\right\| _{L^q}^{\sigma } +\left\| v\right\| _{L^q}^{\sigma })\left\| u-v\right\| _{\dot{H}_{r}^{s}}+
(\left\| u\right\| _{L^q}^{\sigma-1} +\left\| v\right\| _{L^q}^{\sigma-1})
(\left\| u\right\| _{\dot{H}_{r}^{s}}+\left\| v\right\| _{\dot{H}_{r}^{s}})\left\| u-v\right\| _{L^{q}},
\end{split}\end{eqnarray}
for any $1\le k\le [s]$. \eqref{GrindEQ__3_63_} and \eqref{GrindEQ__3_68_} imply that
\begin{eqnarray}\begin{split} \label{GrindEQ__3_69_}
\left\| II_{k} \right\| _{\dot{H}_{p}^{v} }&\lesssim II_{k_1}+II_{k_2}\\
&\lesssim(\left\| u\right\| _{L^q}^{\sigma } +\left\| v\right\| _{L^q}^{\sigma })\left\| u-v\right\| _{\dot{H}_{r}^{s}}+
(\left\| u\right\| _{L^q}^{\sigma-1} +\left\| v\right\| _{L^q}^{\sigma-1})
(\left\| u\right\| _{\dot{H}_{r}^{s}}+\left\| v\right\| _{\dot{H}_{r}^{s}})\left\| u-v\right\| _{L^{q}},
\end{split}\end{eqnarray}
for any $1\le k\le [s]$. This completes the proof in the case $s\notin \mathbb N$.
\end{proof}

\begin{remark}\label{rem 3.3.}
\textnormal{Let $f(z)$ be a polynomial in $z$ and $\bar{z}$ satisfying $\deg(f)=\sigma $.
Using the same argument as in the proof of Lemma \ref{lem 3.2.}, we can easily get
\begin{eqnarray}\begin{split}\nonumber
\left\| f(u)-f(v)\right\| _{\dot{H}_{p}^{s} } \lesssim(\left\| u\right\| _{L^q}^{\sigma } +\left\| v\right\| _{L^q}^{\sigma })\left\| u-v\right\| _{\dot{H}_{r}^{s}}
+(\left\| u\right\| _{L^q}^{\sigma-1} +\left\| v\right\| _{L^q}^{\sigma-1})
(\left\| u\right\| _{\dot{H}_{r}^{s}}+\left\| v\right\| _{\dot{H}_{r}^{s}})\left\| u-v\right\| _{L^{q}},
\end{split}\end{eqnarray}
for $s\ge 1$ and $1<p,r<\infty$, $1<q\le \infty$ satisfying $\frac{1}{p}=\frac{1}{r}+\frac{\sigma}{q}$.}
\end{remark}

\begin{remark}\label{rem 3.4.}
\textnormal{Lemma \ref{lem 3.2.} generalizes Corollary 3.5 of \cite{D18I}, which only covers the case $\sigma\ge \lceil s\rceil$. Lemma \ref{lem 3.2.} also improves Lemmas 3.2 and 3.9 of \cite{AKC22}.
In fact, putting $\frac{1}{q}:=\frac{1}{r}-\frac{s}{d}$, Lemma 3.2 of \cite{AKC22} follows directly from Lemma \ref{lem 3.2.} by using the embedding $\dot{H}_r^s\hookrightarrow L^q$.
Putting $q:=\infty$, Lemma 3.9 of \cite{AKC22} also follows from Lemma \ref{lem 3.2.} by using the embedding $H_r^s\hookrightarrow L^{\infty}$.
Similarly, Lemma \ref{lem 3.1.} generalizes Lemmas 3.1 and 3.8 of \cite{AKC22}}
\end{remark}

We also have the following result which generalizes Lemmas 3.4 and 3.11 of \cite{AKC22}.
\begin{lemma}\label{lem 3.5.}
Let $\sigma >0$ and $1\le p,q, r\le \infty$ with $\frac{1}{p}=\frac{1}{r}+\frac{\sigma}{q}$.
Assume that $f\in C^{1} \left(\mathbb C\to \mathbb C\right)$ satisfies
\begin{equation} \label{GrindEQ__3_70_}
\left|f'(u)\right|\lesssim|u|^{\sigma } ,
\end{equation}
for any $u\in \mathbb C$. Then we have
\[\left\| f(u)-f(v)\right\| _{L^p} \lesssim\left(\left\| u\right\| _{L^q}^{\sigma } +\left\| v\right\| _{L^q}^{\sigma } \right)\left\| u-v\right\| _{L^r}.\]
\end{lemma}
\begin{proof}
Since
\[f(u)-f(v)=\int _{0}^{1}\left[f_{z} \left(v+t\left(u-v\right)\right)\left(u-v\right)+f_{\bar{z}} \left(v+t\left(u-v\right)\right)\left(\bar{u}-\bar{v}\right)\right]dt ,\]
it follows from \eqref{GrindEQ__3_70_} that
\[\left|f(u)-f(v)\right|\lesssim\left(|u|^{\sigma } +|v|^{\sigma } \right)\left|u-v\right|.\]
Hence we have
\begin{eqnarray}\begin{split}\nonumber
\left\| f(u)-f(v)\right\| _{L^p}&\lesssim\left\| \left(|u|^{\sigma } +|v|^{\sigma } \right)\left(u-v\right)\right\| _{L^p} \le \left(\left\| u\right\| _{L^q}^{\sigma } +\left\| v\right\| _{L^q}^{\sigma } \right)\left\| u-v\right\| _{L^r},
\end{split}\end{eqnarray}
this completes the proof.
\end{proof}

\begin{remark}\label{rem 3.5.}
\textnormal{Let $s>0$ and $\sigma>0$. If $\sigma$ is not an even integer, assume that $\sigma\ge \lceil s\rceil-1$. If $s<1$, in addition, suppose further that $\sigma\ge 1$. Then we can see that $f(u)=\lambda |u|^{\sigma}u$ with $\lambda\in \mathbb C$ satisfies the assumptions of Lemmas \ref{lem 3.1.}--\ref{lem 3.5.}. See \cite{CFH11,DYC13} for example.}
\end{remark}
%%%%%%%%%%%%%%%%%%%%%%%%%%%%%%%%%%%%%%%%%%%%%%%%%%%%%%%%%%%%%%%%%%%%%%%%%%%%%
\section{Proof of Theorem \ref{thm 1.2.}}\label{sec 4.}

In this section, we prove Theorem \ref{thm 1.2.}.
First, we recall the following useful fact.
\begin{remark}[\cite{AK212}]\label{rem 4.1.}
\textnormal{Let $1<r<\infty$, $s\ge0$ and $b>0$. Let $\chi\in C_{0}^{\infty}(\mathbb R^{d})$ satisfy $\chi(x)=1$ for $|x|\le 1$ and $\chi(x)=0$ for $|x|\ge 2$.
If $b+s<\frac{d}{r}$, then $\chi(x)|x|^{-b} \in \dot{H}_{r}^{s}$.
If $b+s>\frac{d}{r}$, then $(1-\chi(x))|x|^{-b}\in \dot{H}_{r}^{s}$.}
\end{remark}

Using Remark \ref{rem 4.1.} and the estimates of $f(u)-f(v)$ established in Section \ref{sec 3.}, we have the following estimates of the term $|x|^{-b}|u|^{\sigma}u-|x|^{-b}|v|^{\sigma}v$. Such estimates play the crucial role in proving Theorem \ref{thm 1.2.}.
\begin{lemma}\label{lem 4.2.}
$d\in \mathbb N$, $0< s <\min \{2+\frac{d}{2},\frac{3}{2}d\}$, $0<b<\min\{4,d,\frac{3}{2}d-s,\frac{d}{2}+2-s\}$  and $0<\sigma<\sigma_{c}(s)$. If $\sigma$ is not an even integer, assume \eqref{GrindEQ__1_10_}. Then there exist $(a_{i},b_{i})\in A$ $(i=1,2)$, $(p_{j},q_{j})\in B_{0}$ $(j=1,2,3,4)$ and $\theta_k>0$ $(k=1,2,3)$ such that that following assertions hold.

\textnormal{(1)} If $\sigma$ is an even integer, or if not we assume further that $\sigma \ge \hat{\sigma}_{\star}(s)$, then we have
\begin{eqnarray}\begin{split} \label{GrindEQ__4_1_}
&\left\|\chi(x)|x|^{-b}(|u|^{\sigma}u-|v|^{\sigma}v)\right\|_{L^{a_{1}'}(I,\dot{H}_{b_{1}'}^{s-\gamma_{a_{1}',b_{1}'}-4})}\\
&~~~~~~~~\lesssim |I|^{\theta_1}(\left\|u\right\|_{L^{p_{1}}(I,H_{q_{1}}^{s})}^{\sigma}+\left\|v\right\|_{L^{p_{1}}(I,H_{q_{1}}^{s})}^{\sigma})
    \left\|u-v\right\|_{L^{p_{1}}(I,H_{q_{1}}^{s})},
\end{split}\end{eqnarray}
\begin{eqnarray}\begin{split} \label{GrindEQ__4_2_}
&\left\|(1-\chi(x))|x|^{-b}(|u|^{\sigma}u-|v|^{\sigma}v)\right\|_{L^{a_{2}'}(I,\dot{H}_{b_{2}'}^{s-\gamma_{a_{2}',b_{2}'}-4})}\\
&~~~~~~~\lesssim (|I|^{\theta_2}+|I|^{\theta_3})\max_{j\in\{2,3,4\}}(\left\|u\right\|_{L^{p_{j}}(I,H_{q_{j}}^{s})}^{\sigma}+\left\|v\right\|_{L^{p_{j}}(I,H_{q_{j}}^{s})}^{\sigma})
\max_{j\in\{2,3,4\}}
      \left\|u-v\right\|_{L^{p_{j}}(I,H_{q_{j}}^{s})},
\end{split}\end{eqnarray}
where $I(\subset \mathbb R)$ is an interval and $\chi\in C_{0}^{\infty}(\mathbb R^{d})$ is given in Remark \ref{rem 4.1.}.

\textnormal{(2)} If $\sigma$ is not an even integer satisfying $\hat{\sigma}_{\star}(s)>\sigma \ge \tilde{\sigma}_{\star}(s)$, then we have
\begin{eqnarray}\begin{split} \label{GrindEQ__4_3_}
&\left\|\chi(x)|x|^{-b}(|u|^{\sigma}u-|v|^{\sigma}v)\right\|_{L^{a_{1}'}(I,\dot{H}_{b_{1}'}^{s-\gamma_{a_{1}',b_{1}'}-4})}\\
&~~~~~~~~\lesssim |I|^{\theta_1}(\left\|u\right\|_{L^{p_{1}}(I,H_{q_{1}}^{s})}^{\sigma}+\left\|v\right\|_{L^{p_{1}}(I,H_{q_{1}}^{s})}^{\sigma})
    \left\|u-v\right\|_{L^{p_{1}}(I,H_{q_{1}}^{s})}\\
&~~~~~~~~~~~~~~~~~+|I|^{\tilde{\theta}_1}(\left\|u\right\|_{L^{p_{1}}(I,H_{q_{1}}^{s})}^{\hat{\sigma}_{\star}(s)} +\left\|v\right\|_{L^{p_{1}}(I,H_{q_{1}}^{s})}^{\hat{\sigma}_{\star}(s)}) \left\|u-v\right\|_{L^{\tilde{p}_{1}}(I,H_{\tilde{q}_{1}}^{s-\varepsilon_1})}^{\sigma-\hat{\sigma}_{\star}(s)+1},
\end{split}\end{eqnarray}
\begin{eqnarray}\begin{split} \label{GrindEQ__4_4_}
&\left\|(1-\chi(x))|x|^{-b}(|u|^{\sigma}u-|v|^{\sigma}v)\right\|_{L^{a_{2}'}(I,\dot{H}_{b_{2}'}^{s-\gamma_{a_{2}',b_{2}'}-4})}\\
&~~~~~~~\lesssim (|I|^{\theta_2}+|I|^{\theta_3})\max_{j\in\{2,3,4\}}(\left\|u\right\|_{L^{p_{j}}(I,H_{q_{j}}^{s})}^{\sigma}+\left\|v\right\|_{L^{p_{j}}(I,H_{q_{j}}^{s})}^{\sigma})
\max_{j\in\{2,3,4\}}
      \left\|u-v\right\|_{L^{p_{j}}(I,H_{q_{j}}^{s})}\\
&~~~~~~~~~~~~~~~~~+|I|^{\tilde{\theta}_2}(\left\|u\right\|_{L^{p_{2}}(I,H_{q_{2}}^{s})}^{\hat{\sigma}_{\star}(s)-1} +\left\|v\right\|_{L^{p_{2}}(I,H_{q_{2}}^{s})}^{\hat{\sigma}_{\star}(s)-1}) \left\|u\right\|_{L^{p_{3}}(I,H_{q_{3}}^{s})} \left\|u-v\right\|_{L^{\tilde{p}_{2}}(I,H_{\tilde{q}_{2}}^{s-\varepsilon_2})}^{\sigma-\hat{\sigma}_{\star}(s)+1},
\end{split}\end{eqnarray}
for some
$(\tilde{p}_{j},\tilde{q}_{j})\in B_{0}$, $\tilde{\theta}_{j}>0$, $\varepsilon_{j}>0$ with $s-\varepsilon_j>0$ $(j=1,2)$.
\end{lemma}
\begin{proof} Without loss of generality, we assume that $\sigma$ is not an even integer.
We use the similar argument as in the proof of Lemma 3.2 of \cite{AKR22}.
We divide the study in two cases: $0\le s<\frac{d}{2}$ and $s>\frac{d}{2}$.

{\bf Case 1.} We consider the case $0\le s<\frac{d}{2}$.

First, we prove \eqref{GrindEQ__4_1_} and \eqref{GrindEQ__4_3_}.
We choose the same admissible pairs $(a_{1},b_{1})\in A$ and $(p_{1},q_{1})\in B_{0}$ as in Case 1 in the proof of Lemma 3.2 of \cite{AKR22}, which satisfy the following system:
\begin{equation} \label{GrindEQ__4_5_}
\left\{\begin{array}{l}
{\frac{1}{b_{1}'}-\sigma\left(\frac{1}{q_{1}}-\frac{s}{d}\right)-\left(\frac{1}{q_{1}}-\frac{\gamma_{a_{1}',b_{1}'}+4}{d}\right)>\frac{b}{d},}\\
{\frac{1}{a_{1}'}-\frac{\sigma +1}{p_{1}}>0,}\\
{\frac{1}{q_{1}}-\frac{s}{d}>0,}\\
{s-\gamma_{a_{1}',b_{1}'}-4\ge0,}\\
{\gamma_{a_{1}',b_{1}'}+4\ge 0.}\\
\end{array}\right.
\end{equation}
Putting
\begin{equation} \label{GrindEQ__4_6_}
\left\{\begin{array}{l}
{\frac{1}{\alpha_{1}}:=\frac{1}{q_{1}}-\frac{s}{d},~\frac{1}{\beta_{1}}:=\frac{1}{q_{1}}-\frac{\gamma_{a_{1}',b_{1}'}+4}{d},}\\
{\frac{1}{r_{1}}:=\frac{\sigma}{\alpha_{1}}+\frac{1}{\beta_{1}},~\frac{1}{\gamma _{1}}:=\frac{1}{b_{1}'}-\frac{1}{r_{1}},}\\
{\frac{1}{\bar{r}_{1}}:=\frac{\sigma}{\alpha_{1}},~~\frac{1}{\bar{\gamma}_{1}}:=\frac{1}{b_{1}'}-\frac{1}{\bar{r}_{1}},} \\
\end{array}\right.
\end{equation}
it follows from Lemma \ref{lem 2.1.}, the third and last equations in \eqref{GrindEQ__4_5_} that
\begin{equation} \label{GrindEQ__4_7_}
\dot{H}_{q_{1} }^{s}\hookrightarrow L^{\alpha_{1}}, ~\dot{H}_{q_{1} }^{s}\hookrightarrow\dot{H}_{\beta_{1} }^{s-\gamma_{a_{1}',b_{1}'}-4}.
\end{equation}
Furthermore, it follows from Remark \ref{rem 4.1.} and the first equation in \eqref{GrindEQ__4_5_} that
\begin{equation} \label{GrindEQ__4_8_}
\left\| \chi(x)|x|^{-b} \right\| _{L^{\gamma _{1}}}<\infty, ~\left\| \chi(x)|x|^{-b} \right\| _{\dot{H}_{\bar{\gamma}_{1} }^{s-\gamma_{a_{1}',b_{1}'}-4} }<\infty.
\end{equation}
Hence, using \eqref{GrindEQ__4_6_}--\eqref{GrindEQ__4_8_} and Lemma \ref{lem 3.5.} that
\begin{eqnarray}\begin{split} \label{GrindEQ__4_9_}
&\left\| \chi(x)|x|^{-b} (|u|^{\sigma}u-|v|^{\sigma}v)\right\| _{\dot{H}_{b_{1}'}^{s-\gamma_{a_{1}',b_{1}'}-4}}\\
&~~~\lesssim \left\| \chi(x)|x|^{-b} \right\| _{L^{\gamma _{1}}} \left\| |u|^{\sigma}u-|v|^{\sigma}v\right\| _{\dot{H}_{r_{1} }^{s-\gamma_{a_{1}',b_{1}'}-4} }
          +\left\| \chi(x)|x|^{-b} \right\| _{\dot{H}_{\bar{\gamma}_{1} }^{s-\gamma_{a_{1}',b_{1}'}-4} } \left\| |u|^{\sigma}u-|v|^{\sigma}v\right\| _{L^{\bar{r}_{1} } }\\
&~~~\lesssim \left\| |u|^{\sigma}u-|v|^{\sigma}v\right\| _{\dot{H}_{r_{1} }^{s-\gamma_{a_{1}',b_{1}'}-4} }
          +(\left\| u\right\| _{\dot{H}_{q_{1} }^{s}}^{\sigma } +\left\| v\right\| _{\dot{H}_{q_{1} }^{s}}^{\sigma })
              \left\|u-v\right\| _{\dot{H}_{q_{1} }^{s}}.
\end{split}\end{eqnarray}

First, we consider the case $s\le 1$. Using the fifth equation in \eqref{GrindEQ__4_5_}, we can see that
$$
s-\gamma_{a_{1}',b_{1}'}-4\le s\le 1.
$$
Hence, it follows from the assumption \eqref{GrindEQ__1_10_}, \eqref{GrindEQ__4_6_}, \eqref{GrindEQ__4_7_} and Lemma \ref{lem 3.1.} that
\begin{equation}\nonumber
\left\| |u|^{\sigma}u-|v|^{\sigma}v\right\| _{\dot{H}_{r_{1} }^{s-\gamma_{a_{1}',b_{1}'}-4} }
\le (\left\| u\right\| _{\dot{H}_{q_{1} }^{s}}^{\sigma } +\left\| v\right\| _{\dot{H}_{q_{1} }^{s}}^{\sigma })
              \left\|u-v\right\| _{\dot{H}_{q_{1} }^{s}}.
\end{equation}
Combining the above inequality with \eqref{GrindEQ__4_9_}, \eqref{GrindEQ__4_10_}, the second equation in \eqref{GrindEQ__4_5_} and H\"{o}lder inequality, we get \eqref{GrindEQ__4_1_}.

Next, we consider the case $1<s<\frac{d}{2}$ with $d\ge 3$. Since $\sigma\ge \tilde{\sigma}_{\star}(s)\ge 1$, we can choose $(a_{1},b_{1}):=(2,\frac{2d}{d-2})$ as in the proof of Lemma 3.2 of \cite{AKR22}. Note that $s-\gamma_{a_{1}',b_{1}'}-4=s-1$. If either $s\le 2$ or $s>2$ and $\sigma\ge \lceil s\rceil-1$, we can also get \eqref{GrindEQ__4_1_} by using Lemmas \ref{lem 3.1.}, \ref{lem 3.2.} and the same argument as in the case $s\le 1$. It remains to consider the case $s>2$ and $\lceil s\rceil-1>\sigma\ge \lceil s\rceil-2$. Using \eqref{GrindEQ__4_6_}, \eqref{GrindEQ__4_7_}, \eqref{GrindEQ__4_9_} and Lemma \ref{lem 3.2.}, we have
\begin{eqnarray}\begin{split} \label{GrindEQ__4_10_}
&\left\| \chi(x)|x|^{-b} (|u|^{\sigma}u-|v|^{\sigma}v)\right\| _{\dot{H}_{b_{1}'}^{s-\gamma_{a_{1}',b_{1}'}-4}}\\
&~~~\lesssim(\left\| u\right\| _{\dot{H}_{q_{1} }^{s}}^{\sigma } +\left\| v\right\| _{\dot{H}_{q_{1} }^{s}}^{\sigma })
              \left\|u-v\right\| _{\dot{H}_{q_{1} }^{s}}
          +(\left\| u\right\| _{\dot{H}_{q_{1} }^{s}}^{\hat{\sigma}_{\star}(s)} +\left\| v\right\| _{\dot{H}_{q_{1} }^{s}}^{\hat{\sigma}_{\star}(s)})\left\| u-v\right\| _{L^{\alpha_{1}}}^{\sigma -\hat{\sigma}_{\star}(s)+1}.
\end{split}\end{eqnarray}
Putting $\theta_{1}:=\frac{1}{a_{1}'}-\frac{\sigma +1}{p_{1}}>0$, the second equation in the system \eqref{GrindEQ__4_5_} imply that $\theta_{1}>0$. Putting
$$
\frac{1}{\hat{p}_1}:=\frac{1}{p_1}+\frac{\theta_1}{\sigma-\hat{\sigma}_{\star}(s)+1},
$$
we can see that $0<\hat{p}_{1}<p_1$ and
\begin{equation} \label{GrindEQ__4_11_}
\frac{1}{a_{1}'}=\frac{\hat{\sigma}_{\star}(s)}{p_{1}}+\frac{\sigma-\hat{\sigma}_{\star}(s)+1}{\hat{p}_1}
\end{equation}
Hence, using \eqref{GrindEQ__4_10_}, \eqref{GrindEQ__4_11_} and H\"{o}lder inequality, we immediately get
\begin{eqnarray}\begin{split} \label{GrindEQ__4_12_}
&\left\|\chi(x)|x|^{-b}(|u|^{\sigma}u-|v|^{\sigma}v)\right\|_{L^{a_{1}'}(I,\dot{H}_{b_{1}'}^{s-\gamma_{a_{1}',b_{1}'}-4})}\\
&~~~~~~~~\lesssim |I|^{\theta_1}(\left\|u\right\|_{L^{p_{1}}(I,H_{q_{1}}^{s})}^{\sigma}+\left\|v\right\|_{L^{p_{1}}(I,H_{q_{1}}^{s})}^{\sigma})
    \left\|u-v\right\|_{L^{p_{1}}(I,H_{q_{1}}^{s})}\\
&~~~~~~~~~~~~~~~~~+(\left\|u\right\|_{L^{p_{1}}(I,H_{q_{1}}^{s})}^{\hat{\sigma}_{\star}(s)} +\left\|v\right\|_{L^{p_{1}}(I,H_{q_{1}}^{s})}^{\hat{\sigma}_{\star}(s)}) \left\|u-v\right\|_{L^{\hat{p}_{1}}(I,L^{\alpha_1})}^{\sigma-\hat{\sigma}_{\star}(s)+1},
\end{split}\end{eqnarray}
Since $(p_1,q_1)\in B_{0}$, $0<\hat{p}_{1}<p_1$ and $s>0$, we can take $\eta_1>0$ sufficiently small such that
\begin{equation} \label{GrindEQ__4_13_}
(\tilde{p}_1,\tilde{q}_1)\in B_0,~\hat{p}_{1}<\tilde{p}_1<p_1,~s-\varepsilon_1>0,
\end{equation}
where
\begin{equation} \nonumber
\tilde{q}_1:=q_1+\eta_1,~\varepsilon_1:=\frac{\eta_1d}{q_1(q_1+\eta_1)}.
\end{equation}
Noticing $\frac{1}{\alpha_{1}}=\frac{1}{q_{1}}-\frac{s}{d}=\frac{1}{\tilde{q}_{1}}-\frac{s-\varepsilon_1}{d}$, we have the embedding $\dot{H}_{\tilde{q}_1}^{s-\varepsilon_1}\hookrightarrow L^{\alpha_1}$. Hence, using \eqref{GrindEQ__4_12_}, \eqref{GrindEQ__4_13_} and H\"{o}lder inequality, we have \eqref{GrindEQ__4_3_} with $\tilde{\theta}_1:=\frac{1}{\hat{p}_1}-\frac{1}{\tilde{p}_1}>0$.

Next, we prove \eqref{GrindEQ__4_2_} and \eqref{GrindEQ__4_4_}.
We take the same admissible pairs $(a_{2},b_{2})\in A$ and $(p_{i},q_{i})\in B_{0}$ $(i=2,3,4)$ as in Case 1 in the proof of Lemma 3.2 of \cite{AKR22}, which satisfy the following system:
\begin{equation} \label{GrindEQ__4_14_}
\left\{\begin{array}{ll}
{\frac{1}{q_{2}}-\frac{s}{d}\le\frac{1}{\alpha_{2}}<\frac{1}{q_{2}},\frac{1}{q_{3}}-\frac{\gamma_{a_{2}',b_{2}'}+4}{d}\le\frac{1}{\beta_{2}}\le\frac{1}{q_{3}}},~&~~~~~~~~{(1)}\\
{0<\frac{1}{b_{2}'}-\frac{\sigma}{\alpha_{2}}-\frac{1}{\beta_{2}}<\frac{b}{d},}~&~~~~~~~~{(2)}\\
{\frac{1}{a_{2}'}-\frac{\sigma}{p_{2}}-\frac{1}{p_{3}}>0,}~&~~~~~~~~{(3)}\\
{\frac{1}{q_{2}}>\frac{s}{d},\frac{1}{q_{3}}>\frac{\gamma_{a_{2}',b_{2}'}+4}{d},}~&~~~~~~~~{(4)}\\
{\frac{1}{q_{4}}-\frac{s}{d}\le\frac{1}{\alpha_{4}}\le\frac{1}{q_{4}},}~&~~~~~~~~{(5)}\\
{0<\frac{1}{b_{2}'}-\frac{\sigma+1}{\alpha_{4}}<\frac{b+s-\gamma_{a_{2}',b_{2}'}-4}{d},}~&~~~~~~~~{(6)}\\
{\frac{1}{a_{2}'}-\frac{\sigma+1}{p_{4}}>0,\frac{1}{q_{4}}>\frac{s}{d},}~&~~~~~~~~{(7)}\\
{s-\gamma_{a_{2}',b_{2}'}-4\ge0,\gamma_{a_{2}',b_{2}'}+4\ge 0.}~&~~~~~~~~{(8)}\\
\end{array}\right.
\end{equation}
Using Lemma \ref{lem 2.2.}, Corollary \ref{cor 2.3.} and the equations (1), (4), (5), (7), (8) in the system \eqref{GrindEQ__4_14_}, we have the embeddings:
\begin{equation} \label{GrindEQ__4_15_}
H_{q_{2}}^{s}\hookrightarrow L^{\alpha_{2}},~H_{q_{3}}^{s}\hookrightarrow \dot{H}_{\beta_{2}}^{s-\gamma_{a_{2}',b_{2}'}-4},
~H_{q_{4} }^{s}\hookrightarrow L^{\alpha_{4}}.
\end{equation}
Putting
\begin{equation} \label{GrindEQ__4_16_}
\left\{\begin{array}{l}
{\frac{1}{r_{2}}:=\frac{\sigma}{\alpha_{2}}+\frac{1}{\beta_{2}},~\frac{1}{\gamma _{2}}:=\frac{1}{b_{2}'}-\frac{1}{r_{2}},}\\
{\frac{1}{r_{4}}:=\frac{\sigma+1}{\alpha_{4}},~\frac{1}{\gamma _{4}}:=\frac{1}{b_{2}'}-\frac{1}{r_{4}},}\\
\end{array}\right.
\end{equation}
it follows from Remark \ref{rem 4.1.} and the equations (2), (6) in \eqref{GrindEQ__4_14_} that
\begin{equation} \label{GrindEQ__4_17_}
\left\|(1-\chi(x))|x|^{-b} \right\| _{L^{\gamma _{2}}}<\infty,
~\left\| (1-\chi(x))|x|^{-b}\right\| _{\dot{H}_{\gamma_{4}}^{s-\gamma_{a_{2}',b_{2}'}-4}}<\infty.
\end{equation}
Hence, using \eqref{GrindEQ__4_14_}--\eqref{GrindEQ__4_17_} and Lemma \ref{lem 3.5.}, we have
\begin{eqnarray}\begin{split} \label{GrindEQ__4_18_}
&\left\| (1-\chi(x))|x|^{-b} (|u|^{\sigma}u-|v|^{\sigma}v)\right\| _{\dot{H}_{b_{2}'}^{s-\gamma_{a_{2}',b_{2}'}-4}}\\
&~~~~~~~~~~~~\lesssim \left\| (1-\chi(x))|x|^{-b} \right\| _{L^{\gamma _{2}}} \left\| |u|^{\sigma}u-|v|^{\sigma}v\right\| _{\dot{H}_{r_{2} }^{s-\gamma_{a_{2}',b_{2}'}-4} }\\
&~~~~~~~~~~~~~~~~~~~+\left\| (1-\chi(x))|x|^{-b} \right\| _{\dot{H}_{\gamma_{4} }^{s-\gamma_{a_{2}',b_{2}'}-4} } \left\| |u|^{\sigma}u-|v|^{\sigma}v\right\| _{L^{r_4}}\\
&~~~~~~~~~~~~\lesssim\left\| |u|^{\sigma}u-|v|^{\sigma}v\right\| _{\dot{H}_{r_{2} }^{s-\gamma_{a_{2}',b_{2}'}-4} }+(\left\| u\right\| _{H_{q_{4} }^{s}}^{\sigma } +\left\| v\right\| _{H_{q_{4} }^{s}}^{\sigma }) \left\|u-v\right\| _{\dot{H}_{q_{4} }^{s}}.
\end{split}\end{eqnarray}

If $s\le 1$, we can get \eqref{GrindEQ__4_2_} by using Lemma \ref{lem 3.1.}, \eqref{GrindEQ__4_15_}--\eqref{GrindEQ__4_18_} and the assumption \eqref{GrindEQ__1_10_}.

Let us consider the case $1<s<\frac{d}{2}$ with $d\ge 3$. Since $\sigma\le 1$, we can choose $(a_{2},b_{2}):=(2,\frac{2d}{d-2})$ as in the proof of Lemma 3.2 of \cite{AKR22}. Note also that $s-\gamma_{a_{2}',b_{2}'}-4=s-1$. Hence, if either $s\le 2$ or $s>2$ and $\sigma\ge \lceil s\rceil-1$, we can also get \eqref{GrindEQ__4_2_} by using Lemmas \ref{lem 3.1.}, \ref{lem 3.2.} and the same argument as in the case $s\le 1$. It remains to consider the case $s>2$ and $\lceil s\rceil-1>\sigma\ge \lceil s\rceil-2$. Using \eqref{GrindEQ__4_15_}, \eqref{GrindEQ__4_16_}, \eqref{GrindEQ__4_18_} and Lemma \ref{lem 3.2.}, we have
\begin{eqnarray}\begin{split} \label{GrindEQ__4_19_}
&\left\| (1-\chi(x))|x|^{-b} (|u|^{\sigma}u-|v|^{\sigma}v)\right\| _{\dot{H}_{b_{2}'}^{s-\gamma_{a_{2}',b_{2}'}-4}}\\
&~~~~~~~~~~~~\lesssim(\left\| u\right\| _{H_{q_{2} }^{s}}^{\sigma } +\left\| v\right\| _{H_{q_{2} }^{s}}^{\sigma })
              \left\|u-v\right\| _{\dot{H}_{q_{3} }^{s}}+(\left\| u\right\| _{H_{q_{4} }^{s}}^{\sigma } +\left\| v\right\| _{H_{q_{4} }^{s}}^{\sigma })
              \left\|u-v\right\| _{\dot{H}_{q_{4} }^{s}}\\
&~~~~~~~~~~~~~~~~~~~+(\left\| u\right\| _{H_{q_{2} }^{s}}^{\hat{\sigma}_{\star}(s)-1} +\left\| v\right\| _{H_{q_{2} }^{s}}^{\hat{\sigma}_{\star}(s)-1})\left\| u\right\| _{H_{q_{3} }^{s}}\left\| u-v\right\| _{L^{\alpha_{2}}}^{\sigma -\hat{\sigma}_{\star}(s)+1}.
\end{split}\end{eqnarray}
Putting
$$
\theta_{2}:=\frac{1}{a_{2}'}-\frac{\sigma }{p_{2}}-\frac{1}{p_3}>0,~\theta_{3}:=\frac{1}{a_{2}'}-\frac{\sigma+1}{p_{4}},
$$
the equations (3) and (7) in the system \eqref{GrindEQ__4_14_} imply that $\theta_{2}>0$ and $\theta_{3}>0$.
Putting
$$
\frac{1}{\hat{p}_2}:=\frac{1}{p_2}+\frac{\theta_2}{\sigma-\hat{\sigma}_{\star}(s)+1},
$$
we can see that $0<\tilde{p}_{2}<p_2$ and
\begin{equation} \label{GrindEQ__4_20_}
\frac{1}{a_{2}'}=\frac{\sigma-\hat{\sigma}_{\star}(s)+1}{\hat{p}_2}+\frac{\hat{\sigma}_{\star}(s)-1}{p_{2}}+\frac{1}{p_3}.
\end{equation}
Hence, using \eqref{GrindEQ__4_19_}, \eqref{GrindEQ__4_20_} and H\"{o}lder inequality, we immediately get
\begin{eqnarray}\begin{split} \label{GrindEQ__4_21_}
&\left\|(1-\chi(x))|x|^{-b}(|u|^{\sigma}u-|v|^{\sigma}v)\right\|_{L^{a_{2}'}(I,\dot{H}_{b_{2}'}^{s-\gamma_{a_{2}',b_{2}'}-4})}\\
&~~~~~~~\lesssim (|I|^{\theta_2}+|I|^{\theta_3})\max_{j\in\{2,4\}}(\left\|u\right\|_{L^{p_{j}}(I,H_{q_{j}}^{s})}^{\sigma}+\left\|v\right\|_{L^{p_{j}}(I,H_{q_{j}}^{s})}^{\sigma})
\max_{j\in\{3,4\}}
      \left\|u-v\right\|_{L^{p_{j}}(I,H_{q_{j}}^{s})}\\
&~~~~~~~~~~~~~~~~~+(\left\|u\right\|_{L^{p_{2}}(I,H_{q_{2}}^{s})}^{\hat{\sigma}_{\star}(s)-1} +\left\|v\right\|_{L^{p_{2}}(I,H_{q_{2}}^{s})}^{\hat{\sigma}_{\star}(s)-1}) \left\|u\right\|_{L^{p_{3}}(I,H_{q_{3}}^{s})} \left\|u-v\right\|_{L^{\hat{p}_{2}}(I,L^{\alpha_2})}^{\sigma-\hat{\sigma}_{\star}(s)+1},
\end{split}\end{eqnarray}
Since $(p_2,q_2)\in B_{0}$, $0<\hat{p}_{2}<p_2$ and $s>0$, we can take $\eta_2>0$ sufficiently small such that
\begin{equation} \label{GrindEQ__4_22_}
\tilde{q}_2<\alpha_2,~(\tilde{p}_2,\tilde{q}_2)\in B_0,~\hat{p}_{2}<\tilde{p}_2<p_2,~s-\varepsilon_2>0,
\end{equation}
where
\begin{equation} \nonumber
\tilde{q}_2:=q_2+\eta_2,~\varepsilon_2:=\frac{\eta_2d}{q_2(q_2+\eta_2)}.
\end{equation}
Noticing $\frac{1}{\tilde{q}_{2}}-\frac{s-\varepsilon_2}{d}\le \frac{1}{\alpha_{2}}<\frac{1}{q_{2}}$, we have the embedding $H_{\tilde{q}_2}^{s-\varepsilon_2}\hookrightarrow L^{\alpha_2}$. Hence using \eqref{GrindEQ__4_21_}, \eqref{GrindEQ__4_22_} and H\"{o}lder inequality, we get \eqref{GrindEQ__4_4_} with $\tilde{\theta}_2:=\frac{1}{\hat{p}_2}-\frac{1}{\tilde{p}_2}>0$.

{\bf Case 2.} We consider the case $s\ge\frac{d}{2}$.

First, we prove \eqref{GrindEQ__4_1_} and \eqref{GrindEQ__4_3_}. We also use the same admissible pair $(a_{1},b_{1})\in S$ as in Case 2  in the proof of Lemma 3.2 of \cite{AKR22}, which satisfy
\begin{equation} \label{GrindEQ__4_23_}
\left\{\begin{array}{l}
{2<\alpha_{1},\bar{\alpha}_{1}<\infty,~2\le \beta_{1}\le b_{1}},\\
{\frac{1}{b_{1}'}-\frac{\sigma}{\alpha_{1}}-\frac{1}{\beta_{1}}>\frac{b}{d},~\frac{1}{b_{1}'}-\frac{\sigma}{\bar{\alpha}_{1}}>\frac{b+s-\frac{2}{a_{1}}}{d}}.\\
\end{array}\right.
\end{equation}
Using Lemma \ref{lem 2.2.}, Corollary \ref{cor 2.3.} and the first equation in the system \eqref{GrindEQ__4_23_}, we have the embeddings:
\begin{equation} \label{GrindEQ__4_24_}
H^{s}\hookrightarrow L^{\alpha_{1}},~H^{s}\hookrightarrow L^{\bar{\alpha}_{1}},~H^{s}\hookrightarrow \dot{H}_{\beta_{1}}^{s-\frac{2}{a_{1}}},
\end{equation}
where we used the fact $\frac{2}{a_1}+\frac{d}{b_1}=\frac{d}{2}$.
Putting
\begin{equation} \label{GrindEQ__4_25_}
\left\{\begin{array}{l}
{\frac{1}{r_{1}}:=\frac{\sigma}{\alpha_{1}}+\frac{1}{\beta_{1}},~\frac{1}{\gamma _{1}}:=\frac{1}{b_{1}'}-\frac{1}{r_{1}},}\\
{\frac{1}{\bar{r}_{1}}:=\frac{\sigma+1}{\bar{\alpha}_{1}},~\frac{1}{\bar{\gamma}_{1}}:=\frac{1}{b_{1}'}-\frac{1}{\bar{r}_{1}}.}\\
\end{array}\right.
\end{equation}
it follows from Remark \ref{rem 4.1.} and the second equation in \eqref{GrindEQ__4_23_} that
\begin{equation} \label{GrindEQ__4_26_}
\left\|\chi(x)|x|^{-b} \right\| _{L^{\gamma _{1}}}<\infty,
~\left\| \chi(x)|x|^{-b}\right\| _{\dot{H}_{\bar{\gamma}_{1}}^{s-\frac{2}{a_1}}}<\infty.
\end{equation}
Hence, it follows from \eqref{GrindEQ__4_24_}--\eqref{GrindEQ__4_26_} and Lemma \ref{lem 3.5.} that
\begin{eqnarray}\begin{split} \label{GrindEQ__4_27_}
&\left\| \chi(x)|x|^{-b} (|u|^{\sigma}u-|v|^{\sigma}v)\right\| _{\dot{H}_{b_{1}'}^{s-\frac{2}{a_{1}}}}\\
&~~~~~~~~~\lesssim \left\|\chi(x)|x|^{-b} \right\| _{L^{\gamma _{1}}} \left\| |u|^{\sigma}u-|v|^{\sigma}v\right\|_{\dot{H}_{r_{1}}^{s-\frac{2}{a_{1}}}}
          +\left\| \chi(x)|x|^{-b}\right\| _{\dot{H}_{\bar{\gamma}_{1}}^{s-\frac{2}{a_{1}}}} \left\||u|^{\sigma}u-|v|^{\sigma}v\right\|_{L^{\bar{r}_{1}}}\\
&~~~~~~~~~\lesssim\left\| |u|^{\sigma}u-|v|^{\sigma}v\right\|_{\dot{H}_{r_{1}}^{s-\frac{2}{a_{1}}}}
+(\left\| u\right\| _{H^{s}}^{\sigma } +\left\| v\right\| _{H^{s}}^{\sigma })
              \left\|u-v\right\| _{H^{s}}.
\end{split}\end{eqnarray}
As in the proof of Lemma 3.2 of \cite{AKR22}, we take $b_{1}:=\frac{2d}{d-2}$ if $d\ge 3$, and $b_{1}(>2)$ large enough such that $\lceil s-\frac{d}{2}+\frac{2}{b_1}\rceil=\left[s-\frac{d}{2}\right]$ if $d\le 2$.
If $\sigma\ge \hat\sigma_{\star}(s)$, then we can get \eqref{GrindEQ__4_1_} by using the assumption \eqref{GrindEQ__1_10_}, \eqref{GrindEQ__4_24_}--\eqref{GrindEQ__4_27_}, Lemmas \ref{lem 3.1.} and \ref{lem 3.2.}.
It remains to consider the case $\hat\sigma_{\star}(s)>\sigma\ge \tilde\sigma_{\star}(s)$.
It follows from \eqref{GrindEQ__4_24_}--\eqref{GrindEQ__4_27_}, Lemma \ref{lem 3.2.} and H\"{o}lder inequality that
\begin{eqnarray}\begin{split} \label{GrindEQ__4_28_}
&\left\|\chi(x)|x|^{-b}(|u|^{\sigma}u-|v|^{\sigma}v)\right\|_{L^{a_{1}'}(I,\dot{H}_{b_{1}'}^{s-\frac{2}{a_1}})}\\
&~~~~~~~~\lesssim |I|^{\theta_1}(\left\|u\right\|_{L^{\infty}(I,H^{s})}^{\sigma}+\left\|v\right\|_{L^{\infty}(I,H^{s})}^{\sigma})
    \left\|u-v\right\|_{L^{\infty}(I,H^{s})}\\
&~~~~~~~~~~~~~~~~~+(\left\|u\right\|_{L^{\infty}(I,H^{s})}^{\hat{\sigma}_{\star}(s)} +\left\|v\right\|_{L^{\infty}(I,H^{s})}^{\hat{\sigma}_{\star}(s)}) \left\|u-v\right\|_{L^{\hat{p}_{1}}(I,L^{\alpha_1})}^{\sigma-\hat{\sigma}_{\star}(s)+1},
\end{split}\end{eqnarray}
where $\theta_1:=1-\frac{1}{a_1}$ and $\frac{\sigma-\lceil s\rceil+1}{\hat{p}_1}:=\frac{1}{a_{1}'}$. We then take $\eta_1>0$ sufficiently small such that
\begin{equation} \label{GrindEQ__4_29_}
\tilde{q}_1:=2+\eta_1<\alpha_1,~(\tilde{p}_1,\tilde{q}_1)\in B_0,~\hat{p}_{1}<\tilde{p}_1<\infty,~s-\varepsilon_1>0,
\end{equation}
where $\varepsilon_1:=\frac{\eta_1d}{2(2+\eta_1)}$.
Using \eqref{GrindEQ__4_29_} and Corollary \ref{cor 2.3.}, we have the embedding $H_{\tilde{q}_1}^{s-\varepsilon_1}\hookrightarrow L^{\alpha_1}$. Hence, using \eqref{GrindEQ__4_28_}, \eqref{GrindEQ__4_29_} and H\"{o}lder inequality, we immediately get \eqref{GrindEQ__4_3_} with $\tilde{\theta}_1:=\frac{1}{\hat{p}_1}-\frac{1}{\tilde{p}_1}>0$.

Next, we prove \eqref{GrindEQ__4_2_} and \eqref{GrindEQ__4_4_}. We take the same admissible pairs $(a_{2},b_{2})\in A$ and $(p_{i},q_{i})\in B_{0}$ $(i=2,3,4)$ as in Case 2 in the proof of Lemma 3.2 of \cite{AKR22}, which satisfy the following system:
\begin{equation} \label{GrindEQ__4_30_}
\left\{\begin{array}{ll}
{0<\frac{1}{\alpha_{2}}<\frac{1}{q_{2}},\frac{1}{q_{3}}-\frac{\gamma_{a_{2}',b_{2}'}+4}{d}\le\frac{1}{\beta_{2}}\le\frac{1}{q_{3}}},~&~~~~~~~~{(1)}\\
{0<\frac{1}{b_{2}'}-\frac{\sigma}{\alpha_{2}}-\frac{1}{\beta_{2}}<\frac{b}{d},}~&~~~~~~~~{(2)}\\
{\frac{1}{a_{2}'}-\frac{\sigma}{p_{2}}-\frac{1}{p_{3}}>0,}~&~~~~~~~~{(3)}\\
{\frac{1}{q_{3}}>\frac{\gamma_{a_{2}',b_{2}'}+4}{d},}~&~~~~~~~~{(4)}\\
{0<\frac{1}{\alpha_{4}}\le\frac{1}{q_{4}},}~&~~~~~~~~{(5)}\\
{0<\frac{1}{b_{2}'}-\frac{\sigma+1}{\alpha_{4}}<\frac{b+s-\gamma_{a_{2}',b_{2}'}-4}{d},}~&~~~~~~~~{(6)}\\
{\frac{1}{a_{2}'}-\frac{\sigma+1}{p_{4}}>0,}~&~~~~~~~~{(7)}\\
{s-\gamma_{a_{2}',b_{2}'}-4\ge0,\gamma_{a_{2}',b_{2}'}+4\ge 0.}~&~~~~~~~~{(8)}\\
\end{array}\right.
\end{equation}
Using \eqref{GrindEQ__4_30_} and combing the argument above with that in the proof of Lemma 3.2 of \cite{AKR22}, we can get the desired result, whose proof will be omitted.
\end{proof}
Using Lemma 3.3 of \cite{AKR22} and the fact
$$
\left||u|^{\sigma}u-|v|^{\sigma}v\right|\lesssim (|u|^\sigma+|v|^\sigma)|u-v|,
$$
we immediately have the following result.
\begin{lemma}\label{lem 4.3.}
Let $d\in \mathbb N$, $s\ge 0$, $0<b<\min\left\{4,d\right\}$ and $0<\sigma<\sigma_{c}(s)$.
Then there exist $(a_{i},b_{i})\in B_{0}$ $(i=3,4)$, $(p_{j},q_{j})\in B_{0}$ $(j=5,6,7)$ and $\theta_k>0$ $(k=4,5)$ such that
\begin{equation} \label{GrindEQ__4_31_}
\left\|\chi(x)|x|^{-b}|u|^{\sigma}|u-v|\right\|_{L^{a_{3}'}(I,L^{b_{3}'})}
\lesssim |I|^{\theta_4}\left\|u\right\|^{\sigma}_{L^{p_{5}}(I,H_{q_{5}}^{s})}\left\|u-v\right\|_{L^{p_{6}}(I,L^{q_{6}})},
\end{equation}
\begin{equation} \label{GrindEQ__4_32_}
\left\|(1-\chi(x))|x|^{-b}|u|^{\sigma}|u-v|\right\|_{L^{a_{4}'}(I,L^{b_{4}'})}
\lesssim |I|^{\theta_5}\left\|u\right\|^{\sigma}_{L^{p_{7}}(I,H_{q_{7}}^{s})}\left\|u-v\right\|_{L^{p_{7}}(I,L^{q_{7}})}.
\end{equation}
where $I(\subset \mathbb R)$ is an interval and $\chi\in C_{0}^{\infty}(\mathbb R^{d})$ is given in Remark \ref{rem 4.1.}.
\end{lemma}

Now we are ready to prove Theorem \ref{thm 1.2.}.
\begin{proof}[{\bf Proof of Theorem \ref{thm 1.2.}}]
It follows from Theorem \ref{thm 1.1.} that given $u_0 \in H^{s}$, there exist $T_{\max } ,\, T_{\min } \in (0,\, \infty ]$ such that \eqref{GrindEQ__1_1_} has a unique, maximal solution $u\in C((-T_{\min } ,\;T_{\max }),H^{s})$. There exists $0<T<T_{\max } ,\, T_{\min } $ such that if $u_{0,n} \to u_0$ in $H^{s}$ and if $u_{n} $ denotes the solution of \eqref{GrindEQ__1_1_} with the initial data $u_{0,n}$, then $\left\| u _{0,n} \right\| _{H^{s} } \le 2\left\| u_0 \right\| _{H^{s} } $ for $n$ large, and we have $0<T<T_{\max } (u _{0,n} ),\;T_{\min } (u _{0,n} )$ for all sufficiently large $n$ and $u_{n} $ is bounded in $L^{p}((-T,\;T),\;H_{q}^{s})$ for any $(p,q)\in B$.
Since $u$, $u_{n} $ satisfy integral equations:
\[u(t)=e^{it\Delta^2 }u_0 -i\lambda \int _{0}^{t}e^{i(t-\tau )\Delta^2 } |x|^{-b}|u(\tau)|^{\sigma}u(\tau)d\tau,\]
\[u_{n}(t)=e^{it\Delta^2} u_{0,n} -i\lambda \int _{0}^{t}e^{i(t-\tau)\Delta^2} |x|^{-b} |u_n(\tau)|^{\sigma}u_n(\tau)d\tau ,\]
respectively, we have
\begin{equation} \label{GrindEQ__4_33_}
u_{n}(t)-u(t)=e^{it\Delta^2}(u_{0,n} -u_0)-i\lambda \int _{0}^{t}e^{i(t-\tau )\Delta^2 } |x|^{-b}(|u_n(\tau)|^{\sigma}u_n(\tau)-|u(\tau)|^{\sigma}u(\tau)) d\tau
\end{equation}

Using the standard argument (see \cite{CFH11, DYC13} and Chapter 3 or 4 of \cite{C03}), it suffices to prove that there exits $T>0$ sufficiently small such that as $n\to \infty $,
\begin{equation} \label{GrindEQ__4_34_}
u_{n} \to u ~~\textrm{in}~~ L^{p}((-T,\, T),\, H_{q}^{s} ),
\end{equation}
for every $(p,q)\in B$.
We put $I=(-T,T)$ and
$$
M:=\max_{i=\overline{1,7}}(\left\|u\right\|_{L^{p_{i}}(I, H_{q_{i}}^{s})}+\sup_{n\ge 1}{\left\|u_n\right\|_{L^{p_{i}}(I, H_{q_{i}}^{s})}}),
$$
where $(p_{i},q_{i})\in B_{0}$ are given in Lemmas \ref{lem 4.2.} and \ref{lem 4.3.}.
Using \eqref{GrindEQ__4_33_} and Lemma 2.9 (Strichartz estimates), we have
\begin{eqnarray}\begin{split} \label{GrindEQ__4_35_}
\max_{i=\overline{1,7}}\left\|u_n-u\right\|_{L^{p_{i}}(I,H_{q_{i}}^{s})}
&\lesssim \left\| u_{0,n}-u_{0} \right\| _{H^{s}}
+\left\|\chi(x)|x|^{-b}(|u|^{\sigma}u-|v|^{\sigma}v)\right\|_{L^{a_{1}'}(I,\dot{H}_{b_{1}'}^{s-\gamma_{a_{1}',b_{1}'}-4})}\\
&~+\left\|(1-\chi(x))|x|^{-b}(|u|^{\sigma}u-|v|^{\sigma}v)\right\|_{L^{a_{2}'}(I,\dot{H}_{b_{2}'}^{s-\gamma_{a_{2}',b_{2}'}-4})}\\
&~+\left\|\chi(x)|x|^{-b}(|u|^{\sigma}u-|v|^{\sigma}v)\right\|_{L^{a_{3}'}(I,L^{b_{3}'})}\\
&~+\left\|(1-\chi(x))|x|^{-b}(|u|^{\sigma}u-|v|^{\sigma}v)\right\|_{L^{a_{4}'}(I,L^{b_{4}'})},
\end{split}\end{eqnarray}
where $(a_{i}, b_{i})$ $(i=\overline{1,4})$ are given in Lemmas \ref{lem 4.2.} and \ref{lem 4.3.}.

\textbf{Case 1.} If $\sigma$ is an even integer, or if not we assume further that $\sigma\ge \tilde{\sigma}_{\star}$, it follows from \eqref{GrindEQ__4_35_}, Lemmas \ref{lem 4.2.} and \ref{lem 4.3.} that
\begin{equation}\label{GrindEQ__4_36_}
\max_{i=\overline{1,7}}\left\|u_n-u\right\|_{L^{p_{i}}(I,H_{q_{i}}^{s})} \le C \left\| u_{0,n}-u_{0} \right\| _{H^{s} }+C(2T)^{\theta}M^{\sigma}\max_{i=\overline{1,7}}\left\|u_n-u\right\|_{L^{p_{i}}(I,H_{q_{i}}^{s})},
\end{equation}
where $\theta:=\max_{k=\overline{1,5}}{\theta_{k}}$ if $|2T|\ge 1$, and $\theta:=\min_{k=\overline{1,5}}{\theta_{k}}$ if $|2T|<1$.
If we take $T>0$ such that $(2T)^{\theta}M^{\sigma }\le 1/2$, we can deduce from \eqref{GrindEQ__4_36_} that as $n\to \infty $,
\begin{equation}\label{GrindEQ__4_37_}
\max_{i=\overline{1,7}}\left\|u_n-u\right\|_{L^{p_{i}}(I,H_{q_{i}}^{s})}\lesssim\left\| u_{0,n} -u_0 \right\| _{H^{s} } \to 0.
\end{equation}
Using Lemma \ref{lem 2.9.} (Strichartz estimates), Lemma \ref{lem 4.2.}, Lemma \ref{lem 4.3.} and \eqref{GrindEQ__4_37_}, we have
$$
\left\|u_n-u\right\|_{L^{p}(I,H_{q}^{s})}\lesssim\left\| u_{0,n} -u_0 \right\| _{H^{s} } \to 0,~\textrm{as}~n\to \infty
$$
for any $(p,q)\in B$. So the solution flow is locally Lipschitz.

\textbf{Case 2.} If $\sigma$ is not an even integer satisfying $\tilde{\sigma}_{\star}>\sigma\ge\hat{\sigma}_{\star}$. Using \eqref{GrindEQ__4_35_}, Lemmas \ref{lem 4.2.} and \ref{lem 4.3.}, we have
\begin{eqnarray}\begin{split}\label{GrindEQ__4_38_}
\max_{i=\overline{1,7}}\left\|u_n-u\right\|_{L^{p_{i}}(I,H_{q_{i}}^{s})} &\le C \left\| u_{0,n}-u_{0} \right\| _{H^{s} }+C(2T)^{\theta}M^{\sigma}\max_{i=\overline{1,7}}\left\|u_n-u\right\|_{L^{p_{i}}(I,H_{q_{i}}^{s})}\\
&+C(2T)^{\tilde{\theta}}M^{\sigma}
\max_{i=1,2}\left\|u-v\right\|_{L^{\tilde{p}_{i}}(I,H_{\tilde{q}_{i}}^{s-\varepsilon_i})}^{\sigma-\hat{\sigma}_{\star}(s)+1},
\end{split}\end{eqnarray}
where $\tilde{\theta}:=\max_{k=1,2}{\tilde{\theta}_{k}}$ if $|2T|\ge 1$, and $\tilde{\theta}:=\min_{k=1,2}{\tilde{\theta}_{k}}$ if $|2T|< 1$.
Meanwhile, Theorem \ref{thm 1.1.} shows that $\left\|u-v\right\|_{L^{\tilde{p}_{i}}(I,H_{\tilde{q}_{i}}^{s-\varepsilon_i})}\to 0$ as $n\to \infty$.
Hence, using the same argument as in Case 1, we can see that
$$\left\|u_n-u\right\|_{L^{p}(I,H_{q}^{s})}\to 0~ \textrm{as}~n\to \infty,
$$
for any $(p,q)\in B$. This completes the proof.
\end{proof}

%%%%%%%%%%%%%%%%%%%%%%%%%%%%%%%%%%%%%%%%%%%%%%%%%%%%%%%%%%%%%%%%%%%%%%%%%%%%%
%\section*{References}

\end{document}